\documentclass[reqno]{amsart}

\usepackage{amsmath,amsxtra,amscd,amsthm,amsfonts,amssymb,enumitem,microtype,graphicx,epstopdf,mathrsfs,eucal,comment,tikz-cd,ifdraft,bbm,enumitem,xspace}
\usepackage[hidelinks, pagebackref, bookmarks=false]{hyperref}
\hypersetup{
    colorlinks,
    citecolor=black,
    filecolor=black,
    linkcolor=blue,
    urlcolor=black,
}
\usepackage[nospace,noadjust]{cite}
\usepackage{microtype}
\usepackage{url}

\usepackage{thmtools} 
\usepackage{thm-restate}

\input xy
\xyoption{all}

\numberwithin{equation}{section} 
\numberwithin{figure}{section} 

\makeatletter
\let\c@equation\c@figure
\makeatother

\newtheorem{theorem}[equation]{Theorem}
\newtheorem{corollary}[equation]{Corollary}
\newtheorem{lemma}[equation]{Lemma}
\newtheorem{prop}[equation]{Proposition}

\newtheorem{conjecture}[equation]{Conjecture}

\theoremstyle{remark}
\newtheorem{remark}[equation]{Remark}
\newtheorem{example}[equation]{Example}

\newtheorem*{claim*}{Claim}

\theoremstyle{definition}
\newtheorem{defn}[equation]{Definition}

\newcommand\nc{\newcommand}
\nc\on{\operatorname}
\newcommand\CC{{\mathbb C}}
\newcommand\NN{{\mathbb N}}
\newcommand\BC{{\mathbb C}}
\newcommand\BQ{{\mathbb Q}}
\newcommand\BQbar{\overline\BQ}
\newcommand\RR{{\mathbb R}}

\newcommand\ZZ{{\mathbb Z}}

\newcommand\OO{{\mathcal O}}

\newcommand\fp{{\mathfrak p}}
\newcommand\fq{{\mathfrak q}}
\newcommand\fm{{\mathfrak m}}

\nc\Hom{\on{Hom}}
\newcommand\Aut{\operatorname{Aut}}
\newcommand\End{\operatorname{End}}
\newcommand\ad{\mathrm{ad}}

\nc\Sections{\on{Sections}}
\nc\Sym{\on{Sym}}
\nc\Spec{\on{Spec}}
\nc\Specm{\on{Specm}}
\nc\ul{\underline}
\nc{\dfp}{\overset{\cdot}{\fp}}
\nc{\dfq}{\overset{\cdot}{\fq}}
\nc{\dfm}{\overset{\cdot}{\fm}}
\newcommand\W{W}
\newcommand\wt{\operatorname{wt}}
\newcommand\HH{\mathrm{H}}
\newcommand\B{\mathsf{B}}
\newcommand\D{\mathsf{D}}

\newcommand\Mod{\mathbf{Mod}}
\newcommand\Rep{\mathbf{Rep}}
\newcommand\cts{\mathrm{cts}}
\newcommand\cris{\mathrm{cris}}
\newcommand\st{\mathrm{st}}
\newcommand\pst{\mathrm{pst}}
\newcommand\dR{\mathrm{dR}}
\newcommand\mix{\mathrm{mix}}
\newcommand\pure{\mathrm{pure}}

\newcommand\Aff{\mathbf{Aff}}
\newcommand\Lie{\mathrm{Lie}}

\newcommand\gr{\mathrm{gr}}
\newcommand\dual{*}
\newcommand\Res{\mathrm{Res}}

\newcommand\Fd{E}
\newcommand\Fdbar{\overline E}
\newcommand\Kbar{\overline K}
\newcommand\kbar{\overline k}

\newcommand\Weil{W}
\newcommand\WeilD{{}'\Weil}
\renewcommand\1{\mathbbm{1}} %%I'm not sure what this command was already defined as. Hopefully I haven't screwed up anything.
\newcommand\coker{\operatorname{coker}}

\newcommand\nr{\mathrm{nr}}

\newcommand\isoarrow{\overset\sim\rightarrow}
\newcommand\GL{\mathrm{GL}}
\newcommand\SL{\mathrm{SL}}

\newcommand\GG{\mathbb{G}}
\newcommand\HOM{\operatorname{\underline{Hom}}}

\newcommand\dd{\mathrm{d}}
\newcommand\et{\mathrm{\acute et}}
\newcommand\ab{\mathrm{ab}}
\newcommand\Gal{\operatorname{Gal}}
\newcommand\llbrack{[\![}
\newcommand\rrbrack{]\!]}
\newcommand\llpara{(\!(}
\newcommand\rrpara{)\!)}
\newcommand\hatotimes{\mathbin{\widehat{\otimes}}}
\newcommand\V{\mathcal V}

\newcommand\ambrit[2]{#1}
\newcommand\sz{\ambrit{z}{s}}
\newcommand\ou{\ambrit{o}{ou}}
\usepackage[T1]{fontenc}
\usepackage{lmodern}

\title{Semisimplicity of the Frobenius action on $\pi_1$}
\author{L.\ Alexander Betts and Daniel Litt}

\begin{document}

\maketitle

% !TEX root = explicit_main.tex

\begin{abstract}
We prove that the $\BQ_\ell$-pro-unipotent \'etale fundamental group of a smooth geometrically connected variety over a $p$-adic field satisfies an analogue of the weight--monodromy conjecture, and that the action of a Frobenius is semisimple. We prove these results both in the case $(\ell\neq p)$ and $(\ell=p)$. As a result, we show that certain local Bloch--Kato Selmer schemes appearing in a yet-to-be-formali\sz ed ``Chabauty--Kim method at primes of bad reduction'' are affine spaces.
\end{abstract}

\section{Introduction}

Let $K$ be a finite extension of $\BQ_p$ with residue field of size $q$, and fix a prime $\ell\neq p$. Fix a choice of geometric Frobenius $\varphi_K\in G_K$ and an element $\sigma\in I_K$ of inertia which generates tame inertia. If $X/K$ is a smooth proper variety with good reduction, then the Galois action on the \'etale cohomology $\HH^i_\et(X_{\overline K},\BQ_\ell)$ is unramified, and the eigenvalues of $\varphi_K$ are all $q$-Weil numbers of weight $i$. The strong Tate Conjecture also predicts that the action of $\varphi_K$ should be semisimple.

In general -- $X$ not necessarily smooth, proper, or of good reduction -- the behavi\ou r of the Galois representations $\HH^i_\et(X_{\overline K},\BQ_\ell)$ is less well-understood, and is the subject of several major conjectures.

\begin{conjecture}[Weight--monodromy]\label{conj:w-m}
Let $N$ be the nilpotent endomorphism of $\HH^i_\et(X_{\overline K},\BQ_\ell)$ given by $\frac1e\log(\sigma^e)$ where $\sigma^e$ is any power of $\sigma$ acting unipotently on $\HH^i_\et(X_{\overline K},\BQ_\ell)$. Let
\[
0\leq\W_0\HH^i_\et(X_{\overline K},\BQ_\ell)\leq\W_1\HH^i_\et(X_{\overline K},\BQ_\ell)\leq\dots\leq\W_{2i}\HH^i_\et(X_{\overline K},\BQ_\ell)=\HH^i_\et(X_{\overline K},\BQ_\ell)
\]
be the weight filtration on $\HH^i_\et(X_{\overline K},\BQ_\ell)$ constructed by Deligne \cite[\S6]{deligne:hodge-i}. Then for all $j$, the representation $\gr^\W_j\HH^i_\et(X_{\overline K},\BQ_\ell)$ is \emph{pure} of weight $j$ (see Definition~\ref{def:mixed} for a precise definition)\footnote{There are two competing definitions of the word ``pure'' in the literature: one referring to representations, all of whose Frobenius eigenvalues are Weil numbers of a single weight, and one only requiring this condition on each graded piece of the monodromy filtration. In this paper, we will work exclusively with the latter.}.
\end{conjecture}

\begin{conjecture}[Frobenius-semisimplicity]\label{conj:frob}
$\varphi_K$ acts semisimply on $\HH^i_\et(X_{\overline K},\BQ_\ell)$.
\end{conjecture}

When $X$ is smooth and proper, the first of these conjectures is due to Deligne \cite{deligne-hodge-i,rapoport-zink}; it is a folklore theorem\footnote{A proof of this was outlined to the first author by Tony Scholl; to the best of our knowledge there is no published proof of this fact.} that this implies the general case. The second of these conjectures follows from the strong Tate Conjecture and the Rapoport--Zink spectral sequence when $X$ is smooth and proper with semistable reduction, and de Jong's theory of alterations allows one to extend this to the case of arbitrary reduction. We will shortly explain (Corollary \ref{ref:w-m_implies_frob}) why one should believe Conjecture \ref{conj:frob} in general. There are also analogous conjectures in the case $\ell=p$ regarding the action of the crystalline Frobenius on $\D_\pst(\HH^i_\et(X_{\overline K},\BQ_p))$ (see e.g. \cite[Conjecture 3.27]{mokrane}).

\smallskip

The main aim of this paper is to prove analogues of the above conjectures when the cohomology $\HH^i_\et(X_{\overline K},\BQ_\ell)$ is replaced by the \emph{$\BQ_\ell$-pro-unipotent \'etale fundamental groupoid} $\pi_1^{\BQ_\ell}(X_{\overline K})$ of $X_{\overline K}$. This object, whose formal definition we will recall in \S\ref{s:semisimplicity}, consists of affine $\BQ_\ell$-schemes $\pi_1^{\BQ_\ell}(X_{\overline K};x,y)$ for every $x,y\in X(K)$, endowed with an action of $G_K$. These come endowed with various structure maps; in particular each $\pi_1^{\BQ_\ell}(X_{\overline K};x,x)$ is a pro-unipotent group over $\BQ_\ell$. The Lie algebras $\Lie(\pi_1^{\BQ_\ell}(X_{\overline K};x,x))$ and the affine rings $\OO(\pi_1^{\BQ_\ell}(X_{\overline K};x,y))$ admit natural weight filtrations compatible with their induced Galois actions. We will prove the following two results, along with their analogues in the case $\ell=p$.

\begin{theorem}[Weight--monodromy for the fundamental groupoid]\label{thm:w-m}
Suppose that $X/K$ is smooth and geometrically connected. Then:
\begin{enumerate}
	\item $\gr^\W_{-n}\Lie(\pi_1^{\BQ_\ell}(X_{\overline K};x,x))$ is pure of weight $-n$, for all $n$ and all $x\in X(K)$; and
	\item $\gr^\W_n\OO(\pi_1^{\BQ_\ell}(X_{\overline K};x,y))$ is pure of weight $n$, for all $n$ and all $x,y\in X(K)$.
\end{enumerate}
\end{theorem}

\begin{theorem}[Frobenius-semisimplicity for the fundamental groupoid]\label{thm:frob}
Suppose that $X/K$ is smooth and geometrically connected. Then $\varphi_K$ acts semisimply on $\Lie(\pi_1^{\BQ_\ell}(X_{\overline K};x,x))$ and $\OO(\pi_1^{\BQ_\ell}(X_{\overline K};x,y))$ for all $x,y\in X(K)$.
\end{theorem}

The crystalline version of Theorem~\ref{thm:w-m} already appears in work of Vologodsky \cite[Theorem~26]{vologodsky}, to whom this work owes a great debt; we present two proofs, both different to Vologodsky's and more direct.

\begin{remark}
In \cite[Conjecture 3.10]{chiarellotto-lazda:l-independence}, the authors conjecture a number of independence of $\ell$ results for Weil--Deligne representations associated to unipotent fundamental groups. Our Theorem \ref{thm:frob} above implies that the ``weak'' and ``strong'' versions of these conjectures are equivalent. In particular, Chiarellotto--Lazda prove the weak forms of their conjectures for smooth projective curves over mixed characteristic local fields, which, by Theorem \ref{thm:frob}, implies the strong form of their conjectures. 
\end{remark}

\subsection{Canonical splittings}

The main technical result in this paper is a pure linear algebra lemma (Definition \ref{def:canonical_splitting}), which shows that weight--monodromy conditions can be used to overcome ``extension problems'' when studying Frobenius actions. If $(V,\W_\bullet)$ is a filtered representation of $G_K$, then it is not in general true that semisimplicity of the Frobenius action on the graded pieces $\gr^\W_nV$ implies semisimplicity of the action on $V$ itself. However, if each $\gr^\W_nV$ is pure of weight $n$, we will show that in fact there is a canonical $\varphi_K$-equivariant way to split the $\W_\bullet$-filtration on $V$. In particular, in this case Frobenius-semisimplicity of each $\gr^\W_nV$ does imply Frobenius-semisimplicity of $V$.

As well as proving our Frobenius-semisimplicity Theorem \ref{thm:frob}, this also proves the following relation between the weight--monodromy conjecture and the Frobenius-semisimplicity conjecture.

\begin{corollary}\label{ref:w-m_implies_frob}
Suppose that the weight--monodromy conjecture \ref{conj:w-m} holds for all varieties over $K$, and that the Frobenius-semisimplicity conjecture \ref{conj:frob} holds for all smooth proper varieties over $K$. Then the Frobenius-semisimplicity conjecture \ref{conj:frob} holds for all varieties over $K$.
\begin{proof}
If $X/K$ is a variety, then each $\gr^\W_j\HH^i_\et(X_{\overline K},\BQ_\ell)$ is a subquotient of the degree $j$ \'etale cohomology of a smooth proper variety, and hence Frobenius-semisimple. The canonical $\varphi_K$-equivariant splitting provided by the weight--monodromy conjecture for $X$ implies that $\HH^i_\et(X_{\overline K},\BQ_\ell)$ is Frobenius-semisimple.
\end{proof}
\end{corollary}

The canonical splittings we construct also have further consequences for the fundamental groupoids of smooth geometrically connected varieties, in that they provide, for each $x,y\in X(K)$, a canonical $\varphi_K$-invariant path $p_{x,y}\in\pi_1^{\BQ_\ell}(X_{\overline K};x,y)(\BQ_\ell)^{\varphi_K}$. The analogues of these paths in the case $\ell=p$ already appear in the work of Vologodsky, and in the good reduction case they play a central role in the theory of iterated Coleman integration \cite{besser:coleman}; in the case $\ell\neq p$ these paths were used in \cite{alex-netan} to explicitly calculate the non-abelian Kummer map associated to a smooth hyperbolic curve $X$. Our construction of these canonical paths provides a unified explanation for these phenomena.

We remark that in the archimedean setting, i.e.\ in the category of mixed Hodge structures, the existence of similar canonical splittings is well-known. The weight filtration on a mixed Hodge structure is canonically split over~$\BC$ by the Deligne splitting, and there is even a variant of this splitting which is defined over~$\RR$ \cite[Proposition~2.20]{cattani-kaplan-schmid}. Thus there are also canonical choices of $\RR$-pro-unipotent Betti paths between any two points in a smooth connected variety $X/\BC$. The study of these paths will be taken up by the second author in forthcoming work.

\subsubsection{Applications}

As a further application of our main theorems, we apply these results in Section \ref{s:geometry} to study the geometry of local Bloch--Kato Selmer varieties arising in the Chabauty-Kim program. These are three sub-presheaves $\HH^1_e(G_K,U)\subseteq\HH^1_f(G_K,U)\subseteq\HH^1_g(G_K,U)$ of the continuous Galois cohomology presheaf $\HH^1(G_K,U)$ associated to the $\BQ_p$-pro-unipotent \'etale fundamental group $U=\pi_1^{\BQ_p}(X_{\overline K},\bar x)$, and are all representable by affine schemes over $\BQ_p$ \cite[Proposition~3]{minhyong:siegel}\cite[Lemma~5]{minhyong:selmer}. The relevance of the local Selmer varieties to Diophantine geometry and the Chabauty--Kim method comes via a certain \emph{non-abelian Kummer} or \emph{higher Albanese} map
\[
X(K) \rightarrow \HH^1(G_K,U)(\BQ_p)\,.
\]
When $X$ is a smooth projective curve with good reduction, then the image of this map is contained in $\HH^1_e(G_K,U)=\HH^1_f(G_K,U)$, which is known to be an affine space over $\BQ_p$ \cite[Proposition~1.4]{kim:tangential}. In the absence of properness or good reduction assumptions, the image is contained in $\HH^1_g(G_K,U)$, but not in general in $\HH^1_f$.

Our contribution in this paper is to describe the geometry of $\HH^1_g(G_K,U)$ (for $X$ not necessarily proper, of arbitrary dimension, and with no restrictions on the reduction type). In fact, we find that its geometry is as good as could be hoped: it is \emph{canonically} isomorphic to the product of $\HH^1_e(G_K,U)$ and an explicit (pro-finite-dimensional) vector space $\V(\D_\pst(U))^{\varphi=1}$, and therefore also an affine space. As an illustration, we are able to write down explicit dimension formulae in the case that $X$ is a curve.

%\subsection{Acknowledgments} To be added after the referee process is complete.
% !TEX root = explicit_main.tex

\newcommand\y{y}
\newcommand\G{\mathcal G}
\newcommand\pureN{{\widetilde N}}

\section{Weil--Deligne representations}\label{section:WD-reps}

In this section, we fix a finite extension $K$ of $\BQ_p$ with residue field $k$, along with an algebraic closure $\Kbar$. We write $\Weil_K$ for the Weil group of $K$, i.e.\ the subgroup of the absolute Galois group $G_K$ consisting of elements acting on the residue field $\kbar$ via an integer power of the absolute Frobenius $\sigma\colon x\mapsto x^p$, and we write $v\colon\Weil_K\rightarrow\ZZ$ for the unique homomorphism such that $w\in\Weil_K$ acts on $\kbar$ via $\sigma^{v(w)}$. We fix a geometric Frobenius $\varphi_K\in\Weil_K$, i.e.\ an element such that $v(\varphi_K)=-f(K/\BQ_p)$. The following definition is standard.

\begin{defn}
Let $\Fd$ be a field of characteristic $0$. A \emph{Weil representation} with coefficients in $\Fd$ is a representation $\rho\colon \Weil_K\rightarrow\Aut(V)$ of $\Weil_K$ on a finite dimensional $\Fd$-vector space $V$ on which the inertia group $I_K$ acts through a finite quotient. A \emph{Weil--Deligne representation} with coefficients in $\Fd$ consists of a Weil representation $V$ endowed with an $\Fd$-linear endomorphism $N\in\End(V)$, called the \emph{monodromy operator}, such that
\[
\rho(w)\circ N\circ\rho(w)^{-1} = p^{v(w)}\cdot N
\]
%\[
%N\circ\rho(w) = p^{-v(w)}\cdot\rho(w)\circ N
%\]
for all $w\in\Weil_K$. It follows from this condition that $N$ is necessarily nilpotent.

We denote the category of Weil--Deligne representations by $\Rep_\Fd(\WeilD_K)$. The category $\Rep_\Fd(\WeilD_K)$ has a canonical tensor product making it into a neutral Tannakian category, where the tensor $V_1\otimes V_2$ is endowed with the tensor product $\Weil_K$-action, and with the endomorphism $N_{V_1}\otimes1+1\otimes N_{V_2}$.
\end{defn}

\begin{example}
The Weil--Deligne representation $\Fd(1)$ has underlying vector space $\Fd$, trivial monodromy operator $N$, and the Weil group acts via $w\colon x\mapsto p^{v(w)}x$.
\end{example}

As explained in \cite{fontaine:pot_semi-stables}, Weil--Deligne representations arise naturally from Galois representations, both $\ell$-adic and $p$-adic.

\begin{example}\label{ex:l_not_p}
Let $\ell$ be a prime distinct from $p$, and choose a generator $t\in\BQ_\ell(1)$. Let $t_\ell\colon I_K\rightarrow\BQ_\ell(1)$ denote the $\ell$-adic tame character $w\mapsto\left(\frac{w(p^{1/\ell^n})}{p^{1/\ell^n}}\right)_{n\in\NN}$. Then there is a fully faithful exact $\otimes$-functor
\[
\Rep_{\BQ_\ell,\cts}(G_K)\rightarrow\Rep_{\BQ_\ell}(\WeilD_K)
\]
from the category of continuous $\BQ_\ell$-linear\footnote{One can equally well work with continuous $\Fd$-linear representations for any algebraic extension $\Fd/\BQ_\ell$. However, in our context only $\BQ_\ell$-linear representations will appear, so we restrict to this case.} representations of $G_K$ to the category of Weil--Deligne representations. This functor is defined as follows. If $(V,\rho_0)$ is a continuous $\BQ_\ell$-linear representation of $G_K$, there is an open subgroup $I_L\leq I_K$ acting unipotently on $V$ by Grothendieck's $\ell$-adic Monodromy Theorem. We let $N$ denote the endomorphism of $V$ such that
\[
\rho_0(g) = \exp\left(t^{-1}t_\ell(g)N\right)
\]
for all $g\in I_L$. We define an action $\rho$ of $\Weil_K$ on $V$ by
\[
\rho\left(\varphi_K^ng\right)=\rho_0(\varphi_K^ng)\exp\left(-t^{-1}t_\ell(g)N\right)
\]
for $n\in\ZZ$ and $g\in I_K$. The tuple $(V,\rho,N)$ is the Weil--Deligne representation associated to $V$.

There is an alternative construction of the Weil--Deligne representation due to Fontaine which avoids the choice of Frobenius $\varphi_K$ \cite[\S2.2]{fontaine:pot_semi-stables}. If we write $B$ for the $\BQ_\ell$-linear Tate module of the Tate elliptic curve $\GG_m/p^\ZZ$, then $B(-1)$ is an extension of $\BQ_\ell(-1)$ by $\BQ_\ell$ and hence we may form the direct limit $\B_{\st,\ell}:=\varinjlim\Sym^n(B(-1))$, where the transition maps $\Sym^n(B(-1))\hookrightarrow\Sym^{n+1}(B(-1))$ are induced by the inclusion $\BQ_\ell\hookrightarrow B(-1)$.There is a canonical $\BQ_\ell$-algebra structure on $\B_{\st,\ell}$, namely the polynomial algebra $\BQ_\ell[x]$ where $x\in B(-1)$ is any lift of $t^{-1}\in\BQ_\ell(-1)$, as well as a nilpotent derivation $N=-\frac\dd{\dd x}$. For a continuous $\ell$-adic representation $V$ of $G_K$, one sets
\[
\D_\pst(V):=\varinjlim\nolimits_{L/K}(\B_{\st,\ell}\otimes V)^{I_L},
\]
which is a Weil--Deligne representation with respect to the natural action of $\Weil_K$ and the induced monodromy operator $N$. This construction yields an isomorphic Weil--Deligne representation to that constructed earlier, depending only on the choice of $t$. Specifically, if we choose $x$ to be an eigenvector for the chosen Frobenius $\varphi_K$ (of eigenvalue $q$), then the map $\B_{\st,\ell}\twoheadrightarrow\BQ_\ell$ sending $x$ to $0$ induces an isomorphism $\D_\pst(V)\isoarrow V$ of Weil--Deligne representations. 
\end{example}

\begin{example}\label{ex:l_is_p}
Let $L/K$ be a Galois extension, not necessarily finite, and let~$L_0$ denote the maximal subfield of~$L$ which is unramified over~$\BQ_p$. A \textit{discrete $(\varphi,N,G_{L|K})$-module} \cite[\S4.2.1]{fontaine:semi-stables}  consists of a finite-dimensional $L_0$-vector space $V$ endowed with a $\sigma$-linear Frobenius automorphism $\varphi\colon V\rightarrow V$, an $L_0$-linear endomorphism $N\colon V\rightarrow V$, and a semilinear action $\rho_0\colon G_{L|K}\rightarrow\Aut_{\BQ_p}(V)$ such that:
\begin{itemize}
	\item the action of $G_{L|K}$ has open point-stabilisers and commutes with $\varphi$ and $N$; and
	\item we have $N\circ\varphi = p\cdot\varphi\circ N$.
\end{itemize}
This definition comes from the study of potentially semistable representations in abstract $p$-adic Hodge theory.

Given a discrete $(\varphi,N,G_{L|K})$-module $(V,\varphi,N,\rho_0)$, we obtain an $L_0$-linear Weil--Deligne representation whose underlying vector space and monodromy operator are $V$ and $N$, and whose representation of $\Weil_K$ is given by
\[
\rho(w) = \rho_0(w)\varphi^{-v(w)}.
\]
We thus obtain a faithful, exact and conservative $\otimes$-functor $\Mod(\varphi,N,G_{L|K})\rightarrow\Rep_{L_0}(\WeilD_K)$ from the category of discrete $(\varphi,N,G_{L|K})$-modules to the category of $L_0$-linear Weil--Deligne representations. Precomposing with the Dieudonn\'e functor $\D_{\st,L}\colon\Rep_{\BQ_p,\dR}(G_K)\rightarrow\Mod(\varphi,N,G_{L|K})$, we also obtain an exact $\otimes$-functor from the category of de Rham (=potentially semistable \cite[Th\'eor\`eme~0.7]{berger}) representations to the category of $L_0$-linear Weil--Deligne representations. In the particular case that $L=K$ or $L=\overline K$, we denote the functor $\D_{\st,L}$ simply by $\D_\st$ or $\D_\pst$ respectively; the latter functor is faithful and conservative.
\end{example}

%\begin{example}

%\end{example}

If $P$ is a property of (filtered) Weil--Deligne representations and $\ell$ is a prime, we shall say that a $\BQ_\ell$-linear (filtered) Galois representation $V$ has property $P$ just when its associated Weil--Deligne representation has property $P$; when $\ell=p$ this means we assume that $V$ is de Rham. For example, we say that a Weil--Deligne representation is \textit{semistable} \cite[\S1.3.7]{fontaine:pot_semi-stables} just when the action of $I_K$ is trivial. This corresponds to the usual notions of semistability on $\BQ_\ell$-linear representations, namely unipotence of the $I_K$-action when $\ell\neq p$ and $\B_\st$-admissibility when $\ell=p$. The following two properties --- being \textit{Frobenius-semisimple} and \textit{mixed} --- play a central role in this paper.

\begin{defn}
An $\Fd$-linear Weil--Deligne representation $V$ is said to be \textit{Frobenius-semisimple} just when the action of the geometric Frobenius $\varphi_K$ on $V$ is semisimple, or equivalently just when every element of $\Weil_K$ acts semisimply. For a de Rham $\BQ_p$-linear representation $V$ of $G_K$, this is the same as the action of crystalline Frobenius $\varphi$ on $\D_\pst(V)$ being semisimple (as a $\BQ_p$-linear automorphism).
\end{defn}

\begin{defn}\label{def:mixed}
A \emph{$q$-Weil number} of weight $i$ in an algebraically closed field $\Fdbar$ of characteristic $0$ is an element $\alpha\in\Fdbar$ which is algebraic over $\BQ\subseteq\Fdbar$ and satisfies
\[
|\iota(\alpha)| = q^{i/2}
\]
for every complex embedding $\iota\colon\BQbar\hookrightarrow\CC$.

Given a Weil--Deligne representation $V$ over a characteristic $0$ field $\Fd$, we write $V_{\Fdbar}^i$ for the largest $\varphi_K$-stable subspace of $V_{\Fdbar}$ such that all the generalized eigenvalues of $\varphi_K|_{V_{\Fdbar}^i}$ are $q$-Weil numbers of weight $i$, where~$q$ is the size of the residue field of~$K$. By Galois descent, $V_{\Fdbar}^i$ is the base change of a subspace $V^i$ defined over $\Fd$. It follows from the definition that $N(V^i)\subseteq V^{i-2}$.

We say that a Weil--Deligne representation $V$ is \textit{pure of weight $i$} just when $V=\bigoplus_jV^j$ (i.e.\ all the eigenvalues of $\varphi_K$ are $q$-Weil numbers) and the map $N^j\colon V^{i+j}\isoarrow V^{i-j}$ is an isomorphism for all $j\geq0$. We say that a Weil--Deligne representation $V$ endowed with an increasing filtration
\[
\dots\subseteq\W_i V\subseteq\W_{i+1}V\subseteq\dots
\]
by Weil--Deligne subrepresentations is \emph{mixed} just when $\W_\bullet$ is exhaustive and separated and $\gr^\W_iV$ is pure of weight $i$ for all $i$. The filtration $\W_\bullet V$ is called the \emph{weight filtration} of a mixed Weil--Deligne representation $V$; its set of \emph{weights} is the set
\[
\wt(V) := \{i\in\ZZ\::\:\gr^\W_iV\neq0\}\,.
\]

The collection of mixed Weil--Deligne representations naturally forms a symmetric monoidal category $\Rep_\Fd^\mix(\WeilD_K)$, whose morphisms are filtered maps of Weil--Deligne representations, and whose tensor product $\otimes$ and tensor unit $\1$ are defined in the usual way.

A pure Weil--Deligne representation~$V$ of weight~$i$ can be viewed as a mixed Weil--Deligne representation by endowing it with the filtration where $\W_jV=V$ for~$j\geq i$ and $\W_jV=0$ for~$j<i$.
\end{defn}

%\begin{remark}
%When we refer to a filtered Weil--Deligne representation $V$ as being pure of weight $i$, we mean that it is mixed and $\wt(V)\subseteq\{i\}$, so that the underlying unfiltered representation is pure of weight $i$ in the above sense.
%\end{remark}

\begin{remark}
The subspaces $V^i\subseteq V$ defined in Definition~\ref{def:mixed} are stable under the action of $\Weil_K$ and do not depend on the choice of geometric Frobenius $\varphi_K$. Indeed, for any other geometric Frobenius $\varphi_K'$ there is an $n\in\NN$ such that $\rho(\varphi_K)^n=\rho(\varphi_K')^n$, and we can equivalently describe $V_{\Fdbar}^i$ as the largest subspace of $V_{\Fdbar}$ on which all the generalized eigenvalues of $\rho(\varphi_K)^n$ are $q$-Weil numbers of weight $ni$.
\end{remark}

\begin{remark}
Let $V$ be a filtered de Rham representation of $G_K$ on a finite dimensional $\BQ_p$-vector space, and suppose that $V$ is mixed, which for us means that the $\BQ_p^\nr$-linear Weil--Deligne representation associated to $\D_\pst(V)$ is mixed. Then the $K_0$-linear Weil--Deligne representation associated to $\D_\st(V)$ is also mixed. Indeed, $\BQ_p^\nr\otimes_{K_0}\D_\st(V)$ is the inertia-invariant subspace of the Weil--Deligne representation associated to $\D_\pst(V)$, and hence is mixed since taking invariants under actions of finite groups is exact in characteristic $0$ vector spaces.
\end{remark}

\begin{remark}
In what follows, our Weil--Deligne representations will not necessarily be finite-dimensional, and will sometimes be either ind-finite-dimensional or pro-finite-dimensional (i.e.\ a direct limit or inverse limit of finite-dimensional Weil--Deligne representations, respectively). With suitable adaptations all of the definitions, constructions and results of this section also apply to ind-finite-dimensional and pro-finite-dimensional representations (by functoriality). In fact, all the results of this section except Theorem~\ref{thm:tannakian} are stated so as to be true verbatim for ind-finite-dimensional Weil--Deligne representations, and the corresponding statements in the pro-finite-dimensional case are just the duals.
\end{remark}

The following result is well-known.

\begin{theorem}[cf.\ {\cite[Proposition~20]{vologodsky}}]\label{thm:tannakian}
The category $\Rep_\Fd^\mix(\WeilD_K)$ is a neutral Tannakian category over~$\Fd$, and the forgetful functor $\Rep_\Fd^\mix(\WeilD_K)\rightarrow\Rep_\Fd(\WeilD_K)$ is exact, conservative and compatible with the tensor structure. Morphisms in $\Rep_\Fd^\mix(\WeilD_K)$ are strict for the weight filtration.
\begin{proof}[Proof (sketch)]
Compatibility with the tensor structure is easy to check, and conservativity will be a consequence of exactness, since the forgetful functor reflects zero objects. For the remainder, is suffices to prove that any morphism
\[
f\colon V_1\rightarrow V_0
\]
of mixed representations is strict and that its kernel and cokernel are again mixed when endowed with the subspace and quotient filtrations, respectively.

We begin by proving this in the case that $V_0$ and $V_1$ are pure of weights $i_0$ and $i_1$, respectively, viewed as mixed Weil--Deligne representations as described in Definition~\ref{def:mixed}. In this case, the claim amounts to showing that $f=0$ if $i_0\neq i_1$, and that $\ker(f)$ and $\coker(f)$ are pure of weight $i_0=i_1$ otherwise. If $i_1<i_0$, then $f$ must carry $\W_{i_1}V_1=V_1$ into $\W_{i_1}V_0=0$, so $f=0$ in this case. We suppose henceforth that $i_1\geq i_0$.

Now the functor $V\mapsto V^j$ picking out the weight $j$ generalized eigenspace is exact for all $j$, and hence for all $j\geq0$ we have a commuting diagram
\begin{center}
\begin{tikzcd}
0 \arrow{r} & \ker(f)^{i_1+j} \arrow{r}\arrow{d}{N^j} & V_1^{i_1+j} \arrow{r}\arrow{d}{N^j}[swap]{\wr} & V_0^{i_1+j} \arrow[hook]{d}{N^j} \\
0 \arrow{r} & \ker(f)^{i_1-j} \arrow{r} & V_1^{i_1-j} \arrow{r} & V_0^{i_1-j}
\end{tikzcd}
\end{center}
with exact rows. The middle and right-hand vertical maps are an isomorphism and injective, respectively, by purity of $V_0$ and $V_1$, and hence $N^j\colon\ker(f)^{i_1+j}\isoarrow\ker(f)^{i_1-j}$ is an isomorphism. Thus $\ker(f)$ is pure of weight $i_1$; the dual argument establishes that $\coker(f)$ is pure of weight $i_0$ and we are done in the case $i_0=i_1$.

Finally, in the case $i_1>i_0$, we see from the equal-weight case that the image of $f$ is pure of weight $i_0$, while its coimage is pure of weight $i_1$. But these have the same underlying representation, which is only possible if this is zero (e.g.\ since the weight of a non-zero pure representation is the average weight of its generalized $\varphi_K^{-1}$-eigenvalues). Hence $f=0$ in this case too.

\smallskip

Now we deal with the general case. We view $V_1 \overset f\rightarrow V_0$ as a filtered chain complex in the category of Weil--Deligne representations, with $V_0$ in degree $0$. The associated (homological) spectral sequence \cite[Theorem XI.3.1]{maclane:homology} has first page given by
\[
E^1_{i,j} = 
	\begin{cases}
	\coker\left(\gr^\W_if\right) & \text{if $i+j=0$,} \\
	\ker\left(\gr^\W_if\right) & \text{if $i+j=1$,} \\
	0 & \text{else,}
	\end{cases}
\]
and degenerates to
\[
E^\infty_{i,j} = 
	\begin{cases}
	\gr^\W_i\left(\coker(f)\right) & \text{if $i+j=0$,} \\
	\gr^\W_i\left(\ker(f)\right) & \text{if $i+j=1$,} \\
	0 & \text{else.}
	\end{cases}
\]

The differentials on the first page all vanish, since they are morphisms of pure Weil--Deligne representations whose domain has strictly higher weight than the codomain. The same argument establishes that all differentials on higher pages also vanish, and hence we have $E^1_{i,j}=E^\infty_{i,j}$. In particular, $\gr^\W_i(\ker(f))$ and $\gr^\W_i(\coker(f))$ are both pure of weight $i$ for all $i$, so that $\ker(f)$ and $\coker(f)$ are mixed. Strictness of $f$ also follows from degeneration at the first page, since this ensures that the natural maps $\gr^\W_\bullet\ker(f)\rightarrow\ker\left(\gr^\W_\bullet f\right)$ and $\gr^\W_\bullet\coker(f)\rightarrow\coker\left(\gr^\W_\bullet f\right)$  are isomorphisms.
\end{proof}
\end{theorem}

\begin{prop}[cf.\ {\cite[Lemma~21]{vologodsky}}]\label{mixed-extn}
In a $\W$-strict short exact sequence
\[
0\rightarrow V_1\rightarrow V\rightarrow V_2\rightarrow0
\]
of filtered Weil--Deligne representations, if $V_1$ and $V_2$ are mixed, so too is $V$.
\begin{proof}
Taking $\W$-graded pieces, it suffices to prove that any extension of Weil--Deligne representations $V_1$, $V_2$ which are both pure of weight $i$ is again of weight $i$. The induced sequence
\[
0\rightarrow V_1^j\rightarrow V^j\rightarrow V_2^j\rightarrow0
\]
between the weight $j$ generalized Frobenius eigenspaces is exact for each $j$, and hence we are done by the five-lemma applied to $N^j$.
\end{proof}
\end{prop}

\subsection{The canonical splitting}

In what follows, we will need several basic facts about the structure theory of mixed Weil--Deligne representations, most notably that their weight filtrations have \emph{canonical} splittings compatible with their Weil group actions (which is all we will need in \S\ref{s:semisimplicity}). This will be an immediate consequence of the following lemma describing to what extent one can lift maps between associated gradeds of mixed Weil--Deligne representations.

\begin{lemma}\label{lem:weak_morphisms}
Let $V_1$ and $V_2$ be mixed Weil--Deligne representations, and let $\gr^\W_\bullet f\colon\gr^\W_\bullet V_1\rightarrow\gr^\W_\bullet V_2$ be a morphism of graded Weil--Deligne representations. Then there exists a unique linear map $f\colon V_1\rightarrow V_2$ satisfying the following properties:
\begin{enumerate}
	\item\label{condn:weak1} $f$ is $\Weil_K$-equivariant and preserves the $\W$-filtration;
	\item the associated $\W$-graded of $f$ is the map $\gr^\W_\bullet f$; and
	\item\label{condn:weak2} for every $r>0$, the map
	\[
	\sum_{s=0}^r{r\choose s}(-1)^sN_{V_2}^{r-s}\circ f\circ N_{V_1}^s
	\]
	is $\W$-filtered of degree $-r-1$, i.e.\ takes $\W_iV_1$ into $\W_{i-r-1}V_2$ for every $i$.
\end{enumerate}
Moreover, the assignment $\gr^\W_\bullet f\mapsto f$ is linear, and compatible with composition and tensor products.
\begin{proof}
Let us say that a linear map $f\colon V_1\rightarrow V_2$ is a \emph{weak morphism} just when it satisfies conditions~\eqref{condn:weak1} and~\eqref{condn:weak2} above. In other words, a weak morphism is an element $f\in\W_0\HOM(V_1,V_2)^{\Weil_K}$ such that $N^r(f)\in\W_{-r-1}\HOM(V_1,V_2)$ for all $r>0$, where $N$ denotes the monodromy operator on $\HOM(V_1,V_2)=V_1^\dual\otimes V_2$. It follows from this description that composites and tensor products of weak morphisms are weak morphisms, so it suffices to prove that every morphism $\gr^\W_\bullet f\colon \gr^\W_\bullet V_1\rightarrow\gr^\W_\bullet V_2$ is induced by a unique weak morphism $f\colon V_1\rightarrow V_2$.

To prove this, it suffices to prove that for every mixed Weil--Deligne representation $V$ and every element $\overline f\in\gr^\W_0V^{\Weil_K,N=0}$, there is a unique $f\in\W_0V^{\Weil_K}$ lifting $\overline f$ such that $N^r(f)\in\W_{-r-1}V$ for all $r>0$; applying this to $V=\HOM(V_1,V_2)$ yields the desired result. Let $-i$ denote the lowest weight of $V$ --- if $i\leq0$ then all the weights of $V$ are non-negative and the result is trivial. In general, we proceed by induction on $i$, and write $V$ as an extension
\[
0\rightarrow \gr^\W_{-i}V\rightarrow V\rightarrow\widetilde V\rightarrow0
\]
where the weights of $\widetilde V$ are all $>-i$. It follows from the inductive hypothesis that $\overline f\in\gr^\W_0(V)=\gr^\W_0(\widetilde V)$ has a unique lift to an element $\widetilde f\in\W_0\widetilde V^{\Weil_K}$ such that $N^r(\widetilde f)\in\W_{-r-1}\widetilde V$ for all $r>0$. Since $\widetilde f$ is $\Weil_K$-fixed, it lies in $\widetilde V^0$, so we may further lift~$\widetilde f$ to some $f^0\in V^0$.

Since~$f^0$ is a lift of~$\widetilde f$, we have that $f^0\in\W_0V$ and that $N^r(f^0)\in\W_{-r-1}V$ for all~$r<i$. We also have $N^i(f^0)\in\gr^\W_{-i}V$: if~$i=1$ this follows since $\tilde f=\overline f$ lies in the kernel of $N$ on $\tilde V$ by assumption, while if~$i>1$ this follows since $N^{i-1}(f^0)\in\gr^\W_{-i}V$ already. Since~$f^0\in V^0$, we have $N^i(f^0)\in\gr^\W_{-i}V^{-2i}$, and so by purity there is a unique $f^1\in\gr^\W_{-i}V^0$ such that~$N^i(f^1)=N^i(f^0)$. It follows that~$f:=f^0-f^1$ is the unique element of~$\W_0V^0$ which maps to~$\overline f$ and satisfies $N^r(f)\in\W_{-r-1}V$ for all~$r>0$. Unicity implies that~$f$ is also $\Weil_K$-fixed, so the lemma is proved.
\end{proof}
\end{lemma}

\begin{defn}\label{def:canonical_splitting}
Let $V$ be a mixed Weil--Deligne representation. The \emph{canonical splitting of the weight filtration} is the $\Weil_K$-equivariant linear isomorphism
\[
V \isoarrow \gr^\W_\bullet V
\]
obtained by applying Lemma~\ref{lem:weak_morphisms} to the evident isomorphism $\gr^\W_\bullet V\isoarrow \gr^\W_\bullet\gr^\W_\bullet V$. In other words, it is the $\Weil_K$-equivariant map $f$ uniquely characterised by the fact that it takes $\gr^\W_i V$ into $\W_i V$, and that for any $v_i\in\W_i V$ we have
\[
\sum_{s=0}^r{r\choose s}(-1)^s (\gr^\W_\bullet N)^{r-s}\left(f\left(N^s(v_i)\right)\right)\in\bigoplus_{j\geq r}\gr^\W_{i-j-1}V
\]
for all $r>0$, where $\gr^\W_iN$ denotes the induced monodromy operator on $\gr^\W_i V$.

It follows from Lemma~\ref{lem:weak_morphisms} that this splitting is functorial and compatible with tensor products.
\end{defn}

\begin{example}
Suppose that $V$ is an extension of $\Fd$ by $\Fd(1)$ in the category of Weil--Deligne representations. We endow $V$ with the filtration such that $\W_0V=V$, $\W_{-1}V=\W_{-2}V=\Fd(1)$ and $\W_{-3}V=0$, so that $V$ is a mixed Weil--Deligne representation. $V$ admits a canonical choice of basis $v_0,v_2$, where $v_2\in\Fd(1)$ is the canonical generator and $v_0$ is the unique $\varphi_K$-invariant lift of the canonical generator of $\Fd$. With respect to this basis, the actions of $\varphi_K$ and $N$ are given by the matrices $\begin{pmatrix}1&0\\0&q^{-1}\end{pmatrix}$ and $\begin{pmatrix}0&0\\\lambda&0\end{pmatrix}$, respectively, for some $\lambda\in\Fd$.

It follows from this description that the linear isomorphism $\Fd\oplus\Fd(1)\isoarrow V$ defined by the basis $v_0,v_2$ satisfies the conditions of Lemma~\ref{lem:weak_morphisms}, and hence is the canonical splitting of the weight filtration. Note that this splitting is not a splitting in the category of Weil--Deligne representations when $\lambda\neq 0$ above.
\end{example}

%\begin{remark}
%We've seen above that the canonical splitting from Definition~\ref{def:canonical_splitting} need not be a splitting in the category of Weil--Deligne representations. However, when the weights of $V$ are contained in $\{i-1,i\}$ for some $i$, then the splitting is automatically $N$-equivariant; one sees from this that any extension of a pure Weil--Deligne representation of weight $i$ by a pure representation of weight $i-1$ splits.
%\end{remark}

The following corollary plays a crucial role in our proofs of Theorem~\ref{thm:frob}.

\begin{corollary}[to Definition~\ref{def:canonical_splitting}]\label{cor:semisimple_graded_pieces}
Let $V$ be a mixed Weil--Deligne representation. Then $V$ is Frobenius-semisimple if and only if $\gr^\W_\bullet V$ is Frobenius-semisimple.
\end{corollary}

%Although we won't need any more than Definition~\ref{def:canonical_splitting} in what follows, one can further refine the canonical splitting to an even more structured decomposition of a mixed Weil--Deligne representation.

\subsection{Mixed representations as deformations of pure representations}

The canonical splitting of a mixed Weil--Deligne representation allows us to view every mixed Weil--Deligne representation as a deformation of a pure Weil--Deligne representation. Here, a \emph{pure} Weil--Deligne representation means a graded Weil--Deligne representation whose $i$th graded piece is pure of weight $i$ for all $i$. We will continue to write $\gr^\W_iV$ for the graded pieces of a pure Weil--Deligne representation, which are direct summands of $V$. We will usually write $\pureN$ for the monodromy operator on a pure Weil--Deligne representation.

The sense in which mixed Weil--Deligne representations are deformations of pure ones is made precise in the following definition.

\begin{defn}\label{def:mixing}
Let $V$ be a pure Weil--Deligne representation. A collection of \emph{mixing data} $\underline\delta$ for $V$ consists of $\W$-graded endomorphisms $\delta_r\colon\gr^\W_\bullet V\rightarrow\gr^\W_{\bullet-r} V$ of degree $-r$ for $r>0$, satisfying
\begin{itemize}
	\item $\rho(w)\circ\delta_r\circ\rho(w)^{-1}=p^{v(w)}\cdot\delta_r$ for all $r>0$ and all $w\in\Weil_K$; and
	\item $\ad_\pureN^{r-1}(\delta_r)=0$ for all $r>0$ (so in particular $\delta_1=0$). Here, $\ad_\pureN^{r-1}(-)$ denotes the $(r-1)$-fold iterate of the commutator map $\ad_\pureN(-)=[\pureN,-]$.
\end{itemize}
Given mixing data $\underline\delta^0$ and $\underline\delta^1$ for pure Weil--Deligne representations~$V_0$ and~$V_1$, there are natural mixing data on $V_0\oplus V_1$ and $V_0\otimes V_1$, given by $\delta_r=\delta_r^0\oplus\delta_r^1$ and $\delta_r=\delta_r^0\otimes1+1\otimes\delta_r^1$, respectively.
\end{defn}

Now if $\underline\delta$ is a collection of mixing data for a pure Weil--Deligne representation $V$, we define an associated mixed Weil--Deligne representation $V_{\underline\delta}$, whose underlying Weil representation is $V$, whose $\W$-filtration is the filtration underlying the $\W$-grading on $V$, and whose monodromy operator is $N_{\underline\delta}=N+\sum_{i>0}\delta_i$. It is easy to see that the associated graded of $V_{\underline\delta}$ is $V$, and hence $V$ is indeed a mixed Weil--Deligne representation.

%\begin{remark}\label{rmk:splitting_up_delta}
%Specifying the graded endomorphisms $\delta_r$ is equivalent to specifying a single endomorphism $\delta=\sum_{r>0}\delta_r$ which is $\W$-filtered of degree $\leq-1$, and the two conditions on the $\delta_r$ are equivalent to requiring that $\rho(w)\circ\delta\circ\rho(w)^{-1}=p^{v(w)}\cdot\delta$ and that $\ad_N^{r-1}(\delta)$ is $\W$-filtered of degree $\leq-r-1$ for all $r>0$. We will switch between these two perspectives freely.
%\end{remark}

\begin{lemma}\label{lem:mixing_pure}
The functor $(V,\underline\delta)\mapsto V_{\underline\delta}$ defines a $\otimes$-equivalence from the category of pure Weil--Deligne representations with mixing data to the category of mixed Weil--Deligne representations.
\begin{proof}
Let us describe the inverse functor. Let $V$ be a mixed Weil--Deligne representation, and identify $V$ with its associated graded $\gr^\W_\bullet V$ via the canonical splitting from Definition~\ref{def:canonical_splitting}. For clarity, we will write $\pureN=\gr^\W_\bullet N$ for the monodromy operator on $\gr^\W_\bullet V$ and $N$ for the monodromy operator on $V$. Now $\delta\colon=N-\pureN$ is $\W$-filtered of degree $\leq-1$, so we write $\delta=\sum_{r>0}\delta_r$ with $\delta_r$ $\W$-graded of degree $-r$. Showing that the $\delta_r$ satisfy the conditions in Definition~\ref{def:mixing} is equivalent to showing that $\rho(w)\circ\delta\circ\rho(w)^{-1}=p^{v(w)}\cdot\delta$ for all $w\in\Weil_K$, and that $\ad_\pureN^{r-1}(\delta)$ is $\W$-filtered of degree $\leq-r-1$ for all $r>0$. The first of these follows immediately from the commutation relations for $N$ and $\pureN$.

For the second, we proceed by strong induction on $r>0$. Let $\varepsilon_r\in\End(V)$ be given by
\[
\varepsilon_r = \sum_{s=0}^r{r\choose s}(-1)^s\pureN^{r-s}\circ N^r\,.
\]
An easy calculation verifies that $\varepsilon_1=\delta$ and $\varepsilon_{r+1}=\ad_\pureN(\varepsilon_r)+\delta\circ\varepsilon_r$ for all $r>0$. Thus $\varepsilon_r$ is a linear combination of composites $\ad_\pureN^{s_1-1}(\delta)\circ\ad_\pureN^{s_2-1}(\delta)\circ\dots\circ\ad_\pureN^{s_k-1}(\delta)$ for positive integers $s_i$ summing to $r$, and the coefficient of $\ad_\pureN^{r-1}(\delta)$ in $\varepsilon_r$ is $1$. Assuming for the purposes of induction that $\ad_\pureN^{s-1}(\delta)$ is $\W$-filtered of degree $\leq-s-1$ for all $s<r$, we obtain that $\varepsilon_r\equiv\ad_\pureN^{r-1}(\delta)$ modulo $\W_{-r-1}\End(V)$. But the construction of the canonical splitting ensures that $\varepsilon_r\in\W_{-r-1}$, and so $\ad_\pureN^{r-1}(\delta)\in\W_{-r-1}$ as claimed.

It remains to check that this functor is inverse to the functor $(V,\underline\delta)\mapsto V_{\underline\delta}$. One direction -- that $V=V_{\underline\delta}$ for a mixed Weil--Deligne representation $V$ where $\underline\delta$ is the mixing data constructed above -- is clear. For the other, we wish to show that if $\underline\delta$ is mixing data for a pure Weil--Deligne representation $V$ then the mixing data constructed on $V_{\underline\delta}$ is $\underline\delta$ again. For this, we reverse the above argument: the condition that $\ad_\pureN^{r-1}(\delta)$ is $\W$-filtered of degree $\leq-r-1$ for all $r>0$ ensures that $\sum_{s=0}^r{r\choose s}(-1)^sN^{r-s}\circ N_{\underline\delta}^r$ is $\W$-filtered of degree $\leq-r-1$ for all $r>0$, and hence the identity map $V_{\underline\delta}\isoarrow V$ is the canonical splitting of the weight filtration of $V_{\underline\delta}$. This implies that the mixing data constructed from $V_{\underline\delta}$ is indeed $\underline\delta$.
\end{proof}
\end{lemma}

\subsection{The Tannaka group of mixed Weil--Deligne representations}\label{ss:tannaka}

To conclude this section, we will use the above structure theory of mixed Weil--Deligne representations to give an explicit description of the Tannaka group of mixed Weil--Deligne representations. For this, we write $\Weil^\mix_K$ for the Tannaka group of Weil representations, all of whose generalized Frobenius eigenvalues are $q$-Weil numbers.% This is an affine group-scheme over\footnote{More properly, this group should probably be denoted $\Weil_{K,E}^\mix$. However, one can check that $\Weil_{K,E}^\mix$ is actually the base-change to $E$ of $\Weil_{K,\BQ}^\mix$, so no confusion should be caused by omitting the subscript. The same applies to all other affine group-schemes in this section.} $E$ admitting a map $\Weil_K\rightarrow\Weil_K^\mix(E)$ whose image is Zariski-dense. Since every representation of $\Weil_K^\mix$ admits a functorial and $\otimes$-compatible grading by the weights of $\varphi_K$-eigenvalues, there is a map $\iota\colon\GG_m\hookrightarrow\Weil_K^\mix$ which we refer to as the \emph{Frobenius-weight torus} in $\Weil_K^\mix$. Precisely, if $V^i$ is the largest $\varphi_K$-stable subspace of a $\Weil_K^\mix$-representation $V$ on which the eigenvalues of $\varphi_K$ are $q$-Weil numbers of weight $i$, then $\lambda\in\GG_m$ acts on $V^i$ by multiplication by $\lambda^i$.

\begin{lemma}\label{lem:explicit_tannaka}
Let $\G_\pure$ (resp.\ $\G_\mix$) denote the Tannaka group of pure (resp.\ mixed) Weil--Deligne representations.
\begin{enumerate}
	\item\label{lemitem:G_pure} There is a canonical isomorphism $\G_\pure\cong\SL_2\rtimes\Weil^\mix_K$, where $\Weil_K$ acts on $\SL_2$ via conjugation by $w\mapsto\begin{pmatrix}1&0\\0&p^{v(w)}\end{pmatrix}\in\GL_2$.
	\item\label{lemitem:G_mix} Let $B$ denote the standard $2$-dimensional representation of $\SL_2\rtimes\Weil_K^\mix$, i.e.\ the standard representation of $\SL_2$ with the Weil group acting via $w\mapsto\begin{pmatrix}1&0\\0&p^{v(w)}\end{pmatrix}$. Then there is a canonical isomorphism $\G_\mix\cong\mathcal U\rtimes\G_\pure$, where $\mathcal U$ is the free pro-unipotent group generated by\footnote{By this, we mean that $\mathcal U$ is the pro-unipotent group pro-representing the functor $\mathcal U'\mapsto\varinjlim_N\Hom(\prod_{r=0}^N\Sym^r(B)(1),\Lie(\mathcal U'))$ from (finite-dimensional) unipotent groups to sets.} $\prod_{r\geq0}\Sym^r(B)(1)$ and the action of $\G_\pure=\SL_2\rtimes\Weil_K^\mix$ on $\mathcal U$ is the natural one.
\end{enumerate}
\end{lemma}

\begin{remark}\label{rmk:tannaka_dictionary}
The above lemma says that any mixed Weil--Deligne representation $V$ carries a canonical action of $\mathcal U\rtimes\SL_2\rtimes\Weil_K^\mix$. The relationship between these two structures is as follows.
\begin{enumerate}
	\item The restriction of the action to $\Weil_K^\mix$ is the Weil group action on $V$.
	\item There is a \emph{Frobenius-weight torus} $\GG_m\hookrightarrow\Weil_K^\mix$. The grading on $V$ corresponding to the action of this torus is the grading $V=\bigoplus_iV^i$, i.e.\ the action of $\lambda\in\GG_m$ on $V^i$ is by multiplication by $\lambda^i$. (In fact, this \emph{defines} the Frobenius-weight torus.)
	\item There is a \emph{weight torus} $\GG_m\hookrightarrow\SL_2\rtimes\Weil_K^\mix$ given by $\lambda\mapsto\left(\begin{pmatrix}\lambda^{-1}&0\\0&\lambda\end{pmatrix},\iota(\lambda)\right)$ with $\iota$ the inclusion of the Frobenius-weight torus. The grading on $V$ corresponding to the action of this torus is the canonical splitting of the weight filtration $\W_\bullet$. In particular, the weight filtration on $V$ is the one underlying this grading. One can check that the weight torus is central in $\SL_2\rtimes\Weil_K^\mix$.
	\item The action of $\log(Y)$, where $Y=\begin{pmatrix}1&0\\1&1\end{pmatrix}\in\SL_2$, is the $\W$-graded monodromy operator $\pureN:=\gr^\W_\bullet N$, viewed as an endomorphism of $V$ via the canonical splitting of the weight filtration.
	\item For~$r\geq2$, let $\exp(\delta_r)\in\Sym^{r-2}(B)(1)^{-2}=\Fd(1)$ denote the standard generator $e_1^{r-2}$, where $e_1$ is the first basis vector in $B$. We view $\exp(\delta_r)$ as an element of $\mathcal U$. Then, as the notation suggests, the action of $\log(\exp(\delta_r))$ is equal to $\delta_r$ where $\underline\delta=(\delta_r)_{r\geq2}$ is the mixing data of $V$ as in Definition~\ref{def:mixing}. In particular, the monodromy operator $N$ is given by the action of $\log(Y)+\sum_{r\geq2}\delta_r$.
\end{enumerate}
\end{remark}

%\begin{remark}
%A consequence of the above descriptions, which could be proved by more direct means, is that the Tannaka groups $\G_\pure$ and $\G_\mix$ don't depend on $E$. More precisely, the Tannaka group $\G_{\pure,E}$ (resp.\ $\G_{\mix,E}$) of pure (resp.\ mixed) $E$-linear Weil--Deligne representations is the base-change to $E$ of $\G_{\pure,\BQ}$ (resp.\ $\G_{\mix,\BQ}$).
%\end{remark}

The following lemma explains the relevance of $\SL_2$ in the context of the weight--monodromy condition.

\begin{lemma}\label{lem:sneaky_SL_2}
Let $\mathcal C$ denote the category of finite-dimensional graded vector spaces $V=\bigoplus_jV^j$ endowed with an endomorphism $N\colon V^\bullet\rightarrow V^{\bullet-2}$ of degree $-2$ such that $N^j\colon V^j\isoarrow V^{-j}$ is an isomorphism for all $j\geq0$. Then $\mathcal C$ is neutral Tannakian, with Tannaka group (with respect to the evident \ambrit{fiber}{fibre} functor) canonically isomorphic to $\SL_2$.
\begin{proof}
If $V$ is a representation of $\SL_2$, we endow $V$ with the grading $V=\bigoplus_jV^j$ where $H_\lambda:=\begin{pmatrix}\lambda&0\\0&\lambda^{-1}\end{pmatrix}$ acts on $V^j$ by multiplication by $\lambda^j$. We let $N$ denote the endomorphism given by $\log(Y)$ with $Y=\begin{pmatrix}1&0\\1&1\end{pmatrix}$. From the commutation relation $H_\lambda\circ\log(Y)\circ H_\lambda^{-1}=\lambda^{-2}\log(Y)$ we see that $N$ is graded of degree $-2$, and we see, e.g.\ from the classification of irreducible representations of $\SL_2$, that $N^j\colon V^j\isoarrow V^{-j}$ is an isomorphism for all $j\geq0$.

This construction provides a $\otimes$-functor $F\colon\Rep(\SL_2)\rightarrow\mathcal C$. To show that $F$ is an equivalence, it suffices to show that it induces a bijection between the sets of isomorphism classes of irreducible objects, and that every object of $\mathcal C$ decomposes as a direct sum of irreducible objects. For the first part, if $V_r$ denotes the $r+1$-dimensional irreducible representation of $\SL_2$ ($r\geq0$), then $F(V_r)$ is generated under $N$ by a single element $v$ in degree $r$, subject to the relation $N^{r+1}(v)=0$. This implies that $F(V_r)$ is irreducible, e.g.\ since any proper subobject would have to contain $N^r(v)$ but not $v$. Conversely, if $V$ is an irreducible object of $\mathcal C$, then let $r$ be the greatest integer such that $V^r\neq0$ (so $r\geq0$). If $v\in V^r\setminus\{0\}$ then $N^{r+1}(v)=0$ but $N^r(v)\neq0$, so that $v$ spans a subobject of $V$ isomorphic to $F(V_r)$. It follows that $V=F(V_r)$. Since the objects $F(V_r)$ are clearly non-isomorphic, we have established that $F$ induces a bijection on isomorphism classes of irreducible objects.

To establish that objects of $\mathcal C$ are completely decomposable, take any non-zero object $V$ and let $\overline V=V/U$ be an irreducible quotient, isomorphic to $F(V_r)$ for some $r\geq0$. Let $\overline v\in\overline V^r\setminus\{0\}$ be a highest weight vector and let $v\in V^r$ be any lift of $\overline v$. Since $N^{r+1}(v)\in U^{-r-2}$, there is some $u\in U^{r+2}$ such that $N^{r+1}(v)=N^{r+2}(u)$; replacing the lift $v$ with $v-N(u)$ if necessary, we may assume that $N^{r+1}(v)=0$. But then we see that the quotient map $V\twoheadrightarrow\overline V$ splits, via the map $N^j(\overline v)\mapsto N^j(v)$. Thus $\overline V$ is in fact a direct summand of $V$, which establishes complete reducibility.
\end{proof}
\end{lemma}

\begin{proof}[Proof of Lemma~\ref{lem:explicit_tannaka}]
\eqref{lemitem:G_pure} We describe a canonical $\otimes$-equivalence of categories between the category of pure Weil--Deligne representations and the category of representations of $\SL_2\rtimes\Weil^\mix_K$. In the one direction, if $V$ is a pure Weil--Deligne representation, then the grading $V=\bigoplus_j\left(\bigoplus_i\gr^\W_iV^{i+j}\right)$ and endomorphism $\pureN$ endow $V$ with the structure of an object of the category $\mathcal C$ from Lemma~\ref{lem:sneaky_SL_2}, and hence endows $V$ with an action $\rho\colon\SL_2\rightarrow\GL(V)$. If we let $V^{(n)}$ denote the object of $\mathcal C$ with the same vector space and grading but with endomorphism $p^n\pureN$, then the corresponding action of $\SL_2$ is given by precomposing the original action with conjugation by $\begin{pmatrix}1&0\\0&p^n\end{pmatrix}$. Now the action of an element $w\in\Weil_K$ can be viewed as an isomorphism $V\isoarrow V^{(v(w))}$ in $\mathcal C$, and is hence equivariant for the $\SL_2$-action. Unpacking this, this says that we have the commutation relation
\[
\rho(w)\circ\rho(M)\circ\rho(w)^{-1} = \rho\left(\begin{pmatrix}1&0\\0&p^{v(w)}\end{pmatrix}\cdot M\cdot\begin{pmatrix}1&0\\0&p^{-v(w)}\end{pmatrix}\right)
\]
for all $w\in\Weil_K$ and $M\in\SL_2$, so that the actions of $\Weil_K$ and $\SL_2$ together induce an action of $\SL_2\rtimes\Weil_K^\mix$ on $V$.

In the other direction, given a representation $V$ of $\SL_2\rtimes\Weil_K^\mix$, we obtain by restriction a representation of $\Weil_K$ and a nilpotent endomorphism $\pureN:=\log(Y)$ satisfying $\rho(w)\circ\pureN\circ\rho(w)^{-1}=p^{v(w)}\cdot\pureN$ for all $w\in\Weil_K$ (where $Y$ is as in Remark~\ref{rmk:tannaka_dictionary}). Now the weight torus $\GG_m\hookrightarrow\SL_2\rtimes\Weil_K^\mix$ is central, and so provides a decomposition $V=\bigoplus_i\gr^\W_iV$ in the category of Weil--Deligne representations. It then follows from Lemma~\ref{lem:sneaky_SL_2} that the map $\pureN^j\colon\gr^\W_iV^{i+j}\isoarrow\gr^\W_iV^{i-j}$ is an isomorphism for all $j\geq0$ and all $i$, so that this provides $V$ with the structure of a pure Weil--Deligne representation.

These two constructions are evidently self-inverse.

\eqref{lemitem:G_mix} It follows from Lemma~\ref{lem:mixing_pure} that the category of mixed Weil--Deligne representations is $\otimes$-equivalent to the category of pure Weil--Deligne representations together with morphisms $\Sym^{r-2}(B)(1)\rightarrow\End(V)$ of pure representations for every $r\geq2$ (where $\delta_r$ is the image of the generator in $\Sym^{r-2}(B)(1)^{-2}=\Fd(1)$). Specifying these morphisms is equivalent to specifying a $\G_\pure$-equivariant morphism $\mathcal U\rightarrow\GL(V)$, which provides the desired description of $\G_\mix$.
\end{proof}

\begin{remark}\label{rmk:X}
In \S\ref{s:geometry}, an important role will be played by the action of the element $X=\begin{pmatrix}1&1\\0&1\end{pmatrix}\in\SL_2$, in particular its fixed vectors (the ``highest weight vectors'' of the underlying $\SL_2$-representation). One can verify that the $X$-fixed vectors of a mixed Weil--Deligne representation $V$ are given by
\[
(V^j)^X = \{\text{$v\in V^j$ such that $N^r(v)\in\W_{j-r}$ for all $r\geq0$}\}\,.
\]
Using this, one sees that the mixing data $(\delta_r)_{r>0}$ associated to a mixed Weil--Deligne representation $V$ is $X$-fixed, i.e.\ we have $X\circ\delta_r\circ X^{-1}=\delta_r$ for all $r$. Indeed, this is actually equivalent to the condition that $\ad_\pureN^{r-1}(\delta_r)=0$ (in the presence of the condition that $\rho(w)\circ\delta_r\circ\rho(w)^{-1}=p^{-v(w)}\delta_r$ for all $w\in\Weil_K$, which ensures that $\delta_r\in\gr^\W_{-r}\End(V)^{-2}$).
\end{remark}

\begin{remark}\label{rmk:varphi_actions}
If $V$ is a $\W$-filtered $(\varphi,N,G_{L|K})$-module which is mixed in the sense that its associated $L_0$-linear Weil--Deligne representation is mixed, then all of the constructions in this section are compatible with the crystalline Frobenius $\varphi$: for instance the canonical splitting (Definition~\ref{def:canonical_splitting}) is $\varphi$-invariant, and the mixing operators $\delta_r$ (Definition~\ref{def:mixing}) satisfy $\varphi\circ\delta_r\circ\varphi^{-1}=p^{-1}\delta_r$. Perhaps the easiest way to see this is to observe that $\varphi$ provides an isomorphism of mixed Weil--Deligne representations from $V$ to $V$ with a rescaled monodromy operator $N$, and observe that all of our constructions are functorial and are unchanged (up to appropriate scaling factors) on rescaling $N$.
\end{remark}
% !TEX root = explicit_main.tex

\section{Results on semisimplicity and weight--monodromy}\label{s:semisimplicity}
As before, we fix $K$ a finite extension of $\BQ_p$, with residue field $k$ and ring of integers $\mathscr{O}_K$.  Let $X$ be a geometrically connected $K$-variety, and let $x\in X(K)$ be a rational point. Fixing an algebraic closure $\bar K$ of $K$, we let $\bar x$ be the geometric point of $X$ associated to $x$.
\subsection{The \'etale fundamental group} 
Let $\ell$ be a prime.  We let $\pi_1^\ell(X_{\bar K}, \bar x)$ be the pro-$\ell$ completion of the geometric \'etale fundamental group $\pi_1^\et(X_{\bar K}, \bar x)$.  As $x$ was a rational point of $X$, there is a natural action of $\Gal(\overline{K}/K)$ on $\pi_1^\ell(X_{\bar K}, \bar x)$.

We let $$\mathbb{Z}_\ell\llbrack\pi_1^\ell(X_{\bar K}, \bar x)\rrbrack:=\varprojlim_{\pi_1^\ell(X_{\bar K}, \bar x)\twoheadrightarrow H} \mathbb{Z}_\ell[H]$$ be the group ring of $\pi_1^\ell(X_{\bar k}, \bar x)$, where the inverse limit is taken over all finite $\ell$-groups arising as continuous quotients of $\pi_1^\ell(X_{\bar k}, \bar x)$.  There is a natural augmentation map $$\epsilon\colon \mathbb{Z}_\ell\llbrack\pi_1^\ell(X_{\bar K}, \bar x)\rrbrack\to \mathbb{Z}_\ell$$ (induced by the map $g\mapsto 1$, for $g\in \pi_1^\ell(X_{\bar K}, \bar x)$), and we let $\mathscr{I}$ be the kernel of this map --- the \emph{augmentation ideal}.  As with any group algebra, there is a natural comultiplication map $$\Delta\colon \mathbb{Z}_\ell\llbrack\pi_1^\ell(X_{\bar K}, \bar x)\rrbrack\to \mathbb{Z}_\ell\llbrack\pi_1^\ell(X_{\bar K}, \bar x)\rrbrack\hatotimes \mathbb{Z}_\ell\llbrack\pi_1^\ell(X_{\bar K}, \bar x)\rrbrack$$ sending a group element $g$ to $g\otimes g$. (Here $\hatotimes$ denotes the completed tensor product.)

Finally, we set $$\BQ_\ell\llbrack\pi_1^\ell(X_{\bar K}, \bar x)\rrbrack=\varprojlim_n (\mathbb{Z}_\ell\llbrack\pi_1^\ell(X_{\bar K}, \bar x)\rrbrack/\mathscr{I}^n\otimes \BQ_\ell).$$ The comultiplication map $\Delta$ induces a comultiplication on $\BQ_\ell\llbrack\pi_1^\ell(X_{\bar K}, \bar x)\rrbrack$. We abuse notation to denote the augmentation ideal of $\BQ_\ell\llbrack\pi_1^\ell(X_{\bar K}, \bar x)\rrbrack$ by $\mathscr{I}$.  The category of topological $\BQ_\ell\llbrack\pi_1^\ell(X_{\bar K}, \bar x)\rrbrack$-modules which are finite-dimensional as $\BQ_\ell$-vector spaces is equivalent to the category of  continuous unipotent $\pi_1^\ell(X_{\bar K},\bar x)$-representations on $\BQ_\ell$-vector spaces.  The ring 
$\BQ_\ell\llbrack\pi_1^\ell(X_{\bar K}, \bar x)\rrbrack$ is thus the (topological) opposite Hopf algebra to the ring of functions on the $\BQ_\ell$-pro-unipotent fundamental group of $X$.

If $\bar x_1, \bar x_2$ are two geometric points of $X$, we let $\pi_1^\et(X_{\bar K}; \bar x_1, \bar x_2)$ be the pro-finite set of ``\'etale paths'' from $\bar x_1$ to $\bar x_2$ (that is, the set of isomorphisms between the fiber functors associated to $\bar x_1, \bar x_2$).  This is a (right) torsor for the group $\pi_1^\et(X_{\bar K}; \bar x_1)$; let $\pi_1^\ell(X_{\bar K}; \bar x_1, \bar x_2)$ be the associated (right) torsor for $\pi_1^\ell(X_{\bar K}, \bar x_1)$.  It is easy to check that the natural left action of $\pi_1^\et(X_{\bar K}, \bar x_2)$ on $\pi_1^\et(X_{\bar K}; \bar x_1, \bar x_2)$ descends to a left action of $\pi_1^{\ell}(X_{\bar K}, \bar x_2)$ on $\pi_1^{\ell}(X_{\bar K}; \bar x_1, \bar x_2)$.  Let $$\mathbb{Z}_\ell\llbrack\pi_1^{\ell}(X_{\bar K}; \bar x_1, \bar x_2)\rrbrack:=\varprojlim_{\pi_1^{\ell}(X_{\bar K}; \bar x_1, \bar x_2)\twoheadrightarrow H} \mathbb{Z}_\ell[H],$$ where the inverse limit is taken over all finite sets with a continuous surjection from $\pi_1^{\ell}(X_{\bar K}; \bar x_1, \bar x_2)$.  This is a free right module of rank one over $\mathbb{Z}_\ell\llbrack\pi_1^{\ell}(X_{\bar K}, \bar x_1)\rrbrack$, and thus inherits an $\mathscr{I}$-adic filtration; we let $$\BQ_\ell\llbrack\pi_1^{\ell}(X_{\bar K}; \bar x_1, \bar x_2)\rrbrack:= \varprojlim_n (\mathbb{Z}_\ell\llbrack\pi_1^{\ell}(X_{\bar K}; \bar x_1, \bar x_2)\rrbrack/\mathscr{I}^n\otimes \BQ_\ell).$$  This vector space also has a natural filtration, which we call the $\mathscr{I}$-adic filtration by an abuse of notation, defined by $$\mathscr{I}^n=\on{ker}(\BQ_\ell\llbrack\pi_1^{\ell}(X_{\bar K}; \bar x_1, \bar x_2)\rrbrack\to \mathbb{Z}_\ell\llbrack\pi_1^{\ell}(X_{\bar K}; \bar x_1, \bar x_2)\rrbrack/\mathscr{I}^n\otimes \BQ_\ell).$$

We define a \emph{rational tangential basepoint} of $X$ to be a $K\llpara t\rrpara$-point of $X$; the inclusion $K\hookrightarrow K\llpara t\rrpara$ allows one to view any $K$-point of $X$ as a rational tangential basepoint.  We let $\overline{K\llpara t\rrpara}$ be the usual algebraic closure of $\overline{K}\llpara t\rrpara$, namely the field of Puiseux series $\overline{K}\llpara t^\BQ\rrpara$, which we fix for the rest of this paper.

Now if $x_1, x_2$ are rational tangential basepoints of $X$, the group $\on{Gal}(\overline{K\llpara t\rrpara}/K\llpara t\rrpara)$ acts on the triple $(X_{\overline{K}}, \bar x_1, \bar x_2)$, and hence by functoriality of $\pi_1^\ell(X_{\bar K}; \bar x_1, \bar x_2)$, on $\pi_1^\ell(X_{\bar K}; \bar x_1, \bar x_2)$.  In particular, if $\bar x_1=\bar x_2$ arise from a rational point $x\in X(K)$, we obtain an action of $\on{Gal}(\overline{K}/K)$ on $\pi_1^\ell(X_{\bar K}, \bar x)$.

Suppose $\ell\not=p$. Then we set $$\Pi^\ell(X_{\bar K}; \bar x_1, \bar x_2):= \BQ_\ell\llbrack\pi_1^{\ell}(X_{\bar K}; \bar x_1, \bar x_2)\rrbrack.$$

We give $\Pi^\ell(X_{\bar K}; \bar x_1, \bar x_2)$ the structure of a Weil--Deligne representation as in Example~\ref{ex:l_not_p}. Explicitly, for each $n$, $\Pi^\ell(X_{\bar K}; \bar x_1,\bar x_2)/\mathscr{I}^n$ is a finite-dimensional vector space, and thus naturally admits the structure of a Weil--Deligne representation as in Example~\ref{ex:l_not_p}. The construction is functorial, giving $\Pi^\ell(X_{\bar K}; \bar x_1,\bar x_2)$ the structure of a pro-finite-dimensional Weil--Deligne representation.

\subsection{The crystalline setting} Suppose $\ell=p$ (the residue characteristic of $K$), and let $x_1, x_2$ be rational tangential basepoints of $X$. Then we set $$\Pi^p(X_{\bar K}; \bar x_1, \bar x_2):=\varprojlim_n \D_\pst( \mathbb{Z}_p\llbrack\pi_1^{p}(X_{\bar K}; \bar x_1, \bar x_2)\rrbrack/\mathscr{I}^n\otimes \BQ_p).$$

\begin{remark}
The object $\Pi^p(X_{\bar K}; \bar x_1, \bar x_2)$ may be interpreted in terms of the log-crystalline fundamental group of the special fiber of a semi-stable model of $(\overline{X}, D)$, but we will not need this interpretation here.
\end{remark}

We abuse notation and set $$\mathscr{I}^n:=\ker(\Pi^p(X_{\bar K}; \bar x_1, \bar x_2)\to \D_\pst( \mathbb{Z}_p\llbrack\pi_1^{p}(X_{\bar K}; \bar x_1, \bar x_2)\rrbrack/\mathscr{I}^n\otimes \BQ_p).$$

As in Example~\ref{ex:l_is_p}, $\Pi^p(X_{\bar K}; \bar x_1, \bar x_2)$ has the structure of a (pro-finite-dimensional) Weil--Deligne representation.

\begin{remark}\label{de-rham-remark}
We briefly explain why the Galois representation $\mathbb{Z}_p\llbrack\pi_1^p(X_{\bar K}; \bar x_1, \bar x_2)\rrbrack/\mathscr{I}^n\otimes \BQ_p$ is de Rham --- this is proven in \cite[Lemma 7.1]{motivic_anabelian_heights} if $\bar x_1, \bar x_2$ arise from rational points of $X$, but does not appear in the literature if these geometric points arise from rational tangential basepoints. We will require this for the proof of Theorem \ref{weight-monodromy-pi1}.

Deligne and Goncharov \cite[Proposition 3.4]{deligne-goncharov} construct a local system on $X\times X$ whose fiber at a point $(\bar x_1, \bar x_2)$ is $(\mathbb{Z}_p\llbrack\pi_1^p(X_{\bar K}; \bar x_1, \bar x_2)\rrbrack/\mathscr{I}^n\otimes \BQ_p)^\vee$, as a higher direct image of a sheaf on a diagram of schemes over $X\times X$ (strictly speaking, Deligne and Goncharov work in the Betti setting, but an identical construction works in the \'etale setting). The fiber of this local system at a $\overline{K}$-point of $X\times X$ is de Rham by e.g. \cite[Lemma 7.1]{motivic_anabelian_heights} (or by \cite[Theorem 1.4]{andreatta-iovita-kim} in the case this point lies on the diagonal of $X\times X$). 

Hence this local system is de Rham in the sense of \cite{riemann-hilbert1} by \cite[Theorem 1.3]{riemann-hilbert1}. Now the result at tangential basepoints follows from the results of \cite[Section 4.3]{riemann-hilbert2}, for example.
\end{remark}

\subsection{The main theorems}
Let $X$ be a smooth geometrically connected variety over $K$, and $x_1, x_2$ rational tangential basepoints of $X$.  Let $\ell$ be a prime, which may be equal to $p$.

The main theorems of this section are:
\begin{theorem}[Weight--monodromy]\label{weight-monodromy-pi1}
The Weil--Deligne representation $\Pi^\ell(X_{\bar K}; \bar x_1, \bar x_2)$, with the canonical weight filtration (Definition~\ref{weight-defn} below), is mixed.
\end{theorem}
\begin{theorem}[Semisimplicity]\label{semi-simplicity-l-adic}
Each element of $W_K$ acts semisimply on the Weil--Deligne representation $\Pi^\ell(X_{\bar K}; \bar x_1, \bar x_2)$.
\end{theorem}

Theorem~\ref{semi-simplicity-l-adic} above admits the following down-to-earth reformulation. If $\ell\not=p$, the theorem says that every Frobenius element of $\text{Gal}(\overline K/K)$ acts semisimply on $\mathbb{Z}_\ell\llbrack\pi_1^\ell(X_{\bar K}; \bar x_1, \bar x_2)\rrbrack$, or equivalently that every element of $W_K$ acts semisimply on $\Pi^\ell(X_{\bar K}; \bar x_1, \bar x_2)/\mathscr{I}^n$ for all $n$ (with the structure of a Weil--Deligne representation given by Example~\ref{ex:l_not_p}). If $\ell=p$, the theorem is the analogous statement for a $K$-linear power of the crystalline Frobenius.  In both cases, the statement is equivalent to the semi-simplicity of the geometric Frobenius $\varphi_K$ fixed at the beginning of Section~\ref{section:WD-reps}.

As an immediate corollary, we have
\begin{corollary}\label{lie-algebra-corollary}
Let $x$ be a rational tangential basepoint of $X$. Let $\mathfrak{g}_X$ be the Lie algebra of the $\BQ_\ell$-pro-unipotent completion of $\pi_1^\et(X_{\bar K}, \bar x)$. Then $(\ell\not=p)$ $\mathfrak{g}_X$ is a mixed Weil--Deligne representation (with respect to the weight filtration defined below) and each element of $\Weil_K$ acts semisimply on it, and $(\ell=p)$ $\D_\pst(\mathfrak{g}_X)$ is mixed and each element of $\Weil_K$ acts semisimply on it.
\end{corollary}
\begin{proof}
The Lie algebra $\mathfrak{g}_K$ may be identified as the set of primitive elements in $\Pi^\ell(X_{\bar K}; \bar x, \bar x)$, i.e.~ the kernel of the map $$\Delta - \on{id}\otimes 1 - 1\otimes \on{id},$$ with the weight filtration inherited from $\Pi^\ell(X_{\bar K}; \bar x, \bar x)$ (Definition \ref{weight-defn}). The result is immediate.
\end{proof}

\subsubsection{Preliminaries} Before giving the proof, we will need to recall some lemmas, most of which are likely well-known to experts.
\begin{prop}\label{H1-abelianization}
There is a canonical (Galois-equivariant) isomorphism $$\pi_1^\ell(X_{\bar K}, \bar x_1)^\ab\overset{\sim}{\longrightarrow} \mathscr{I}/\mathscr{I}^2.$$
Moreover $\pi_1^\ell(X_{\bar K}, \bar x_1)^\ab/\pi_1^\ell(X_{\bar K}, \bar x_1)^\ab[\ell^\infty]\simeq \HH^1(X_{\bar K, \et}, \mathbb{Z}_\ell)^\vee$ canonically (in particular, as Galois modules).
\end{prop}
\begin{proof}
See \cite[Proposition 2.4]{litt}.
\end{proof}
We will also need the following part of the Weight--Monodromy Conjecture.
\begin{prop}\label{Hi-mixed}
Let $Y$ be any smooth $K$-variety with $\dim(Y)\leq 2$. Let $i\in\mathbb{Z}$ and let $\W_\bullet$ be the weight filtration on $\HH^i(Y_{\bar K, \et}, \BQ_\ell)$ \cite{deligne-hodge-i}. Then the Weil--Deligne representation $\HH^i(Y_{\bar K, \et}, \BQ_\ell)$ is mixed with positive weights.
\end{prop}
\begin{proof}
Let $\overline{Y}$ be a simple normal crossings compactification of $Y$, which exists by resolution of singularities.

For $\ell\not=p$, the case of smooth projective $Y$ with semistable reduction is proven for $\ell\not=p$ in \cite[Satz 2.13]{rapoport-zink}; the case $\ell=p$ is proven by Mokrane \cite[Corollaire 6.2.3]{mokrane} (Mokrane proves the log-crystalline statement; the statement here follows by applying the $p$-adic comparison theorem \cite[Theorem 0.2]{tsuji}). The case of arbitrary smooth proper $Y$ follows from Chow's lemma and de Jong's theory of alterations. Finally, the general case follows immediately from the Deligne spectral sequence, i.e.~the Leray spectral sequence associated to the embedding $Y\hookrightarrow \overline{Y}$ (see e.g. \cite[pg.~2]{jannsen-weights}), using Theorem~\ref{thm:tannakian} and Proposition~\ref{mixed-extn}.
\end{proof}

\begin{prop}\label{semisimple-H1}
Let $Y$ be any smooth $K$-variety.
\begin{enumerate}
\item $(\ell\not=p)$: $\varphi_K$ acts semi-simply on $\HH^1(Y_{\bar K, \et}, \BQ_\ell)$.  
\item $(\ell=p)$: $\varphi_K$ acts semi-simply on $\D_\pst(\HH^1(Y_{\bar K, \et}, \BQ_p))$.
\end{enumerate}
\end{prop}
\begin{proof}
We may reduce to the case $Y$ is quasi-projective by replacing $Y$ with an affine open. By the Lefschetz hyperplane theorem, it suffices to prove this for $Y$ a curve.   Let $\overline{Y}$ by the smooth compactification of $Y$; after extending $K$, we may assume $\overline{Y}$ has semistable reduction, and that $\overline{Y}\setminus Y$ is a disjoint union of rational points of $Y$.

Then the result for $\overline{Y}$ smooth proper is immediate from the Rapoport--Zink spectral sequence (see \cite[Satz 2.10]{rapoport-zink} for the case $\ell\not=p$ and \cite[3.23]{mokrane} for the case $\ell=p$, again using the $p$-adic comparison theorem \cite[Theorem 0.2]{tsuji} to apply the statement) and the analogous fact for abelian varieties. To deduce the result for $Y$, note that $$\HH^1(\overline{Y}_{\bar K}, \BQ_\ell)\simeq \W_1\HH^1(Y_{\bar K}, \BQ_\ell)\to \HH^1(Y_{\bar K}, \BQ_\ell)$$ splits $\varphi_K$-equivariantly by Definition~\ref{def:canonical_splitting}; but $\HH^1(Y_{\bar K}, \BQ_\ell)/\W_1\HH^1(Y_{\bar K}, \BQ_\ell)$ is isomorphic to a direct sum of copies of $\BQ_\ell(-1)$, so the result follows.
\end{proof}
\subsection{Weight-monodromy for $\pi_1$}
We are now ready to prove Theorem~\ref{weight-monodromy-pi1}.

Let $Y$ be any smooth geometrically connected $K$-variety with simple normal crossings compactification $\overline{Y}$; let $x_1, x_2$ be rational tangential basepoints of $Y$. 

\begin{defn}[Weight filtration on $\Pi^\ell$]\label{weight-defn}
We define the weight filtration on $\Pi^\ell(Y_{\bar K}; \bar x_1, \bar x_2)$.  Let $$\mathscr{K}=\ker(\Pi^\ell(Y_{\bar K}; \bar x_1, \bar x_2)\to \Pi^\ell(\overline{Y}_{\bar K}; \bar x_1, \bar x_2\rrpara$$ $$\W_{-1}=\mathscr{I}$$ $$\W_{-2}=\mathscr{I}^2+\mathscr{K}$$ and in general, $$\W_{-i}\Pi^\ell(Y_{\bar K}; \bar x_1, \bar x_2)=\sum_{p+q=i, p,q>0} \W_{-p}\Pi^\ell(Y_{\bar K}; \bar x_1, \bar x_2)\cdot \W_{-q}\Pi^\ell(Y_{\bar K}; \bar x_2, \bar x_2)\text{ for } i>2.$$
\end{defn}

\begin{remark}\label{remark-weights}
As with any weight filtration arising in algebraic geometry, we claim the filtration defined above is uniquely characterized as follows. Let $R\subset K((t))$ be a finitely-generated $\mathbb{Z}$-algebra, $\mathscr{Y}$ an $R$-model of $Y$ (that is, a flat $R$-scheme equipped with an isomorphism $\mathscr{Y}_{K((t))}\simeq Y_{K((t))}$). After possibly enlarging $R$, we may let $y_1, y_2$ be $R$-points of $Y$ so that the fiber functors associated to $y_{i,\overline{K((t))}}$ are Galois-equivariantly isomorphic to those associated to those associated to $\bar x_i$. Then there exists an open subset $U$ of $\on{Spec}(R[1/\ell])$ such that for any closed point $\mathfrak{p}$ of $U$, the associated Frobenius element acts on $\gr_i^\W\Pi^\ell(Y_{\bar k(\mathfrak{p})};  y_{1,\bar k(\mathfrak{p})} , \bar y_{2, \bar k(\mathfrak{p})})$ with eigenvalues $\#\bar k(\mathfrak{p})$-Weil numbers of weight~$i$. Note that we spread out a model over $K((t))$, rather than over $K$, to deal with the case where the $x_i$ are rational tangential basepoints.

In particular, the induced filtration on $\mathscr{I}/\mathscr{I}^2$ agrees with the usual weight filtration coming from the identification with $\HH^1(Y_{\bar K}, \BQ_\ell)^\vee$ in Proposition~\ref{H1-abelianization} by construction; then the claim above follows by the multiplicativity of the weight filtration. 

We will use below the resulting compatibility with another description of the weight filtration arising from work of Deligne and Goncharov \cite{deligne-goncharov}.
\end{remark}
We now prove Theorem~\ref{weight-monodromy-pi1}. Before beginning the proof, we will recall the following crucial result of Deligne and Goncharov. Strictly speaking, Deligne and Goncharov only prove this in the topological setting but the proof works identically in the algebraic setting; we state the result in the form we need.
\begin{theorem}[{\cite[Proposition 3.4]{deligne-goncharov}}]\label{thm:deligne-goncharov-recollection}
Let $X$ be a smooth, geometrically connected variety over $K$, and let $\ell$ be a prime. Let $a, b$ be geometric points of $X$, and for $i\in  \{0, 1, \cdots, n\}$ let $Y_i\subset X^n$ be the subvariety given by $$Y_0:=\{(x_1, \cdots, x_n)\subset X^n\mid x_0=a\},$$ $$Y_i:=\{(x_1, \cdots, x_n)\subset X^n\mid x_i=x_{i+1} \text{ if } 1\leq i\leq n-1, \text{ and }$$ $$Y_n:=\{(x_1, \cdots, x_n)\subset X^n\mid x_n=b\}.$$ For $I\subset\{0, 1\cdots, n\}$ let $$Y_N:=\bigcap_{i\in I} Y_I.$$ Let $j_I: Y_I\to X^n$ be the natural inclusion and let $V_I:=(j_I)_*\mathbb{Q}_\ell$ be the pushforward of the constant lisse sheaf from $Y_I$ to $X^n$. Then there is a natural complex $$\mathscr{K}_{a, b}: \mathbb{Q}_\ell\to \bigoplus_{i\in \{0, \cdots, n\}} V_i\to \cdots \to \bigoplus_{I\subset \{0,\cdots, n\}, |I|=p} V_I\to \cdots \to V_{\{0,\cdots, n\}}\to 0$$
such that
\begin{itemize}
\item $\mathbb{H}^i(X^n, \mathscr{K}_{a,b})=0$ for $i<n$
\item $\mathbb{H}^n(X^n, \mathscr{K}_{a,b})= (\mathbb{Z}_\ell\llbrack\pi_1^\ell(X_{\bar K}; a, b)\rrbrack/\mathscr{I}^{n+1}\otimes \BQ_\ell)^\vee$
\end{itemize}
where $\mathbb{H}^i$ denotes the hypercohomology of the complexes above.
\end{theorem}
\begin{proof}[Proof of Theorem~\ref{weight-monodromy-pi1}]
We explain how to deduce the theorem from Proposition~\ref{Hi-mixed}. By Chow's lemma, we may assume $X$ is quasi-projective.

First, note that by the Lefschetz hyperplane theorem \cite[pg.~195]{goresky-macpherson} for fundamental groups, we may reduce to the case where $\dim(X)\leq 2$.

In the case $\ell\not=p$, recall from \cite[Proposition 3.4]{deligne-goncharov} (or Theorem \ref{thm:deligne-goncharov-recollection}) that $(\Pi^\ell(X_{\bar K}; \bar x_1, \bar x_2)/\mathscr{I}^n)^\vee$ may be computed as the hypercohomology of a complex of sheaves on $X^{n-1}$; each of these sheaves is a direct sum of sheaves of the form $j_*\BQ_\ell,$ where $j\colon X^m\to X^{n-1}$ is a closed embedding. Thus, by Proposition~\ref{Hi-mixed}, there is a spectral sequence whose $E^1$ term consists of mixed Weil--Deligne representations of the form $\HH^i(X^m, \mathbb{Z}_\ell)$, where the weight filtration comes from the usual weight filtration on cohomology \cite{deligne-hodge-i} (using the K\"unneth formula), and whose $E^\infty$ page has on it $\on{gr}_{\mathscr{F}^\bullet}(\Pi^\ell(X_{\bar K}; \bar x_1, \bar x_2)/\mathscr{I}^n)^\vee$ for some filtration $\mathscr{F}^\bullet$. Now we may conclude the result by Theorem~\ref{thm:tannakian} and Proposition~\ref{mixed-extn}. 

In the case $\ell=p$, we may conclude once we know that the spectral sequence indeed converges after applying $\D_\pst$, which follows as $\mathbb{Z}_p\llbrack\pi_1^p(X_{\bar K}; \bar x_1, \bar x_2)\rrbrack/\mathscr{I}^n\otimes \BQ_p$ is de Rham (hence potentially semistable) by Remark \ref{de-rham-remark}.
\end{proof}
\subsection{Semisimplicity}
We now begin preparations for the proof of Theorem~\ref{semi-simplicity-l-adic}.  The canonical splitting of the weight filtration from Definition~\ref{def:canonical_splitting} induces a $\Weil_K$-equivariant splitting of the natural quotient map $$\Pi^\ell(Y_{\bar K}; \bar x_1, \bar x_2)\to \Pi^\ell(Y_{\bar K}; \bar x_1, \bar x_2)/\mathscr{I}.$$
\begin{defn}[Canonical Paths]\label{defn:can-paths}
We denote the image of $1$ under this splitting by $p(x_1, x_2)$. This is $\Weil_K$-invariant by construction -- in the case $\ell=p$ it is moreover invariant under the crystalline Frobenius.
\end{defn}

\begin{prop}
Let $x_1, x_2, x_3$ be rational points or rational tangential basepoints of $X$.  Then 
\begin{enumerate}
\item\label{proppart:identity_paths} $p(x_1, x_1)=1$, and
\item\label{proppart:path-composition} $p(x_2, x_3) \circ p(x_1, x_2)=p(x_1, x_3)$.
\end{enumerate}
\end{prop}
\begin{proof}
\eqref{proppart:identity_paths} is immediate from the definition; \eqref{proppart:path-composition} follows from compatibility with tensor products.
\end{proof}

\begin{remark}
In the case $\ell=p$, the paths $p(x_1,x_2)$ are Vologodsky's canonical $p$-adic paths \cite[Proposition~29]{vologodsky}. In the case $\ell\neq p$, the paths $p(x_1,x_2)$ are the canonical $\ell$-adic paths $\gamma_{x_1,x_2}^{\mathrm{can}}$ from \cite[Remark~2.2.5]{alex-netan}.

%The paths $p(x_1, x_2)$ are $\ell$-adic analogues of Vologodsky's canonical $p$-adic paths \cite[Proposition~29]{vologodsky}.
\end{remark}

\begin{proof}[Proof of Theorem~\ref{semi-simplicity-l-adic}]
This is a more involved variant of \cite[Theorem 2.12]{litt}.

By the Lefschetz hyperplane theorem for fundamental groups, we may without loss of generality assume $\dim(X)\leq 1$. Indeed, there exists a smooth curve $C\hookrightarrow X$ such that the induced map on fundamental groups is a surjection; it suffices to prove the theorem for $C$. So we assume $\dim(X)=1$ and let $\overline{X}$ be the connected smooth proper curve compactifying $X$. Let $D=\overline{X}\setminus X$.

Without loss of generality (by replacing $K$ with a finite extension) we may assume $D=\{x_1, \dots, x_n\}$, with the $x_i\in \overline{X}(K)$ rational points of $\overline{X}$.

Recall that if $\ell\not=p$, we have fixed a Frobenius element $\varphi_K\in \text{Gal}(\bar K/K)$; if $\ell=p$, we let $\varphi_K=\varphi^{f(K/\BQ_p)}$ be the smallest power of the crystalline Frobenius which is $K$-linear (so the geometric Frobenius of the underlying Weil--Deligne representation).

We first claim that it suffices to prove the theorem when $x_1=x_2$. Indeed, suppose we know the theorem for $x_1$. Then composition with $p(x_1, x_2)$ is a $\varphi_K$-equivariant isomorphism $$\Pi^\ell(X_{\bar K}, \bar x_1)\overset{\sim}{\longrightarrow}\Pi^\ell(X_{\bar K}; \bar x_1, \bar x_2).$$ So $\varphi$ acts semisimply on $\Pi^\ell(X_{\bar K}; \bar x_1, \bar x_2)$.  Hence we may and do assume $x_1=x_2=x$ for the rest of the proof.

We now claim it suffices to show that the quotient map $$\mathscr{I}\to \mathscr{I}/\mathscr{I}^2$$ splits $\varphi_K$-equivariantly.  Indeed, let $s\colon \mathscr{I}/\mathscr{I}^2\to \mathscr{I}$ be such a splitting; then the map $$\bigoplus_n (\mathscr{I}/\mathscr{I}^2)^{\otimes n}\overset{\bigoplus s^{\otimes n}}{\longrightarrow} \Pi^\ell(X_{\bar K}; \bar x_1, \bar x_2)$$ has dense image.  But $\varphi_K$ acts semi-simply on $(\mathscr{I}/\mathscr{I}^2)$ by Propositions~\ref{H1-abelianization} and~\ref{semisimple-H1}, so we may conclude the theorem.

We now construct such a $\varphi_K$-equivariant splitting $s$.

{\bf Step 1.} We first construct a splitting of the quotient map $$\mathscr{I}\to \mathscr{I}/\W_{-2}.$$  But this map splits $\varphi_K$-equivariantly by the formula in Definition~\ref{def:canonical_splitting}; choose any $\varphi_K$-equivariant splitting $s_1$.

{\bf Step 2.} We now construct a splitting of the map $$\W_{-2}\Pi^\ell(X_{\bar K}, \bar x)\to \W_{-2}/\mathscr{I}^2.$$

For each $i=1, \dots, n$, we choose a rational tangential basepoint $y_i\in X(K\llpara t\rrpara)$ so that the associated $K\llbrack t\rrbrack$-point of $\overline{X}$ (obtained via the valuative criterion for properness) specializes to $x_i$.  Let $p_i=p(x, y_i)$ be the canonical path arising from Definition~\ref{defn:can-paths}.  

$(\ell\not=p)$: Let $\gamma$ be a topological generator of $$\Gal(\overline{K\llpara t\rrpara}/\overline{K}\llpara t\rrpara)^\ell=\pi_1^\ell(\on{Spec}(\overline{K}\llpara t\rrpara, \on{Spec}(\overline{K\llpara t\rrpara}))=\mathbb{Z}_\ell(1).$$  Then the maps $$\iota_i\colon y_i\to X$$ induce maps $$\iota_{i*}\colon \pi_1^\ell(\on{Spec}(\overline{K}\llpara t\rrpara, \on{Spec}(\overline{K\llpara t\rrpara}))\to \pi_1^{\ell}(X_{\bar K}, \bar y_i).$$  

Let $\gamma_i$ be the image of $\iota_{i*}(\gamma)-1$ in $\W_{-2}/\mathscr{I}^2$.  Then the map $$\iota_*\colon \BQ_\ell(1)^{\{x_1, \cdots, x_n\}}\to \W_{-2}/\mathscr{I}^2$$ $$(a_1,\cdots, a_n)\gamma\mapsto \sum a_i \gamma_i$$ is surjective and $\varphi_K$-equivariant; as $\varphi_K$ acts semisimply on $\BQ_\ell(1)^{\{x_1, \cdots, x_n\}}$, $\iota_*$ splits $\varphi_K$-equivariantly, so it suffices to construct a $\varphi_K$-equivariant map $$\tilde s_2\colon \BQ_\ell(1)^{\{x_1, \dots, x_n\}}\to \mathscr{I}$$ such that the diagram
$$\xymatrix{
& \BQ_\ell(1)^{\{x_1, \cdots, x_n\}} \ar[d]^{\iota_*} \ar[dl]_{\tilde s_2}\\
\W_{-2} \ar@{->>}[r] & \W_{-2}/\mathscr{I}^2
}$$ 
commutes.  

Set $\tilde s_2$ to be the map $$\tilde s_2\colon (a_1, \cdots, a_n)\gamma\mapsto \sum_{i=1}^n a_i p_i \cdot \log(\iota_*\gamma)\cdot p_i^{-1}$$

A direct computation shows that this gives the desired section; see the proof of \cite[Theorem 2.12, p.~621-622]{litt} for an identical computation.

Finally, let $p$ be any section to $\iota_*$; then we set $s_2=\tilde s_2\circ p$.

$(\ell=p)$: Let $\beta$ be a topological generator of $\mathbb{Z}_p(1)$ and let $$\Pi^p(y_i):=\varprojlim_n \D_\st(\mathbb{Z}_p\llbrack\mathbb{Z}_p(1)\rrbrack/(\beta-1)^n\otimes \BQ_p).$$

Now the Weil--Deligne representation $K(1):=\D_\st(\BQ_p(1\rrpara$ has $\varphi_K$-action given by multiplication by $q^{-1}=(\# k)^{-1}$ and $N\equiv 0$. As $(\beta-1)/(\beta-1)^2\simeq \mathbb{Z}_p(1)$, there is a $\varphi$-equivariant isomorphism $$\Pi^p(y_i)\simeq \prod_{i\geq 0} K(i),$$ where $K(i)$ is the $\varphi_K$-module $K(1)^{\otimes i}$. Let $\gamma$ be an element of $\Pi^p(y_i)$ such that $\varphi_K(\gamma)=q\gamma$. Note that this element $\gamma$ plays the role of $\log(\gamma)$ in the $(\ell\not=p)$ case of the group above; in particular, it is a primitive element of the Hopf algebra $\Pi^p(y_i)$ rather than a group-like element.

The map $\iota_i$ induces a map $$\iota_{i*}\colon \Pi^p(y_i)\to \Pi^p(X, \bar y_i);$$ let $\gamma_i=\iota_{i*}(\gamma).$

Now the map $$\iota_*\colon K(1)^{\{x_1, \cdots, x_n\}}\to \W_{-2}/\mathscr{I}^2$$ $$(a_1, \cdots, a_n)\mapsto \sum a_i \gamma_i$$

is surjective and $\varphi_K$-equivariant; as $\varphi_K$ acts semisimply on $K(1)^{\{x_1, \cdots, x_n\}}$, $\iota_*$ splits $\varphi_K$-equivariantly, so it suffices to construct a $\varphi_K$-equivariant map $$\tilde s_2\colon K(1)^{\{x_1, \cdots, x_n\}}\to \mathscr{I}$$ such that the diagram
$$\xymatrix{
& K(1)^{\{x_1, \cdots, x_n\}} \ar[d]^{\iota_*} \ar[dl]_{\tilde s_2}\\
\W_{-2} \ar@{->>}[r] & \W_{-2}/\mathscr{I}^2
}$$ 
commutes.  

Set $\tilde s_2$ to be the map $$\tilde s_2\colon (a_1, \cdots, a_n)\gamma\mapsto \sum_{i=1}^n a_i p_i \cdot\gamma_i\cdot p_i^{-1}$$
Again, this gives the desired section by an argument identical to the proof of \cite[Theorem 2.12, p.~621-622]{litt}.

Finally, let $p$ be any section to $\iota_*$; then we set $s_2=\tilde s_2\circ p$.

{\bf Step 3.}  We now construct the desired $\varphi_K$-equivariant section $$s\colon \mathscr{I}/\mathscr{I}^2\to \mathscr{I}.$$namely, $$s\colon v\mapsto s_1(v\bmod \W_{-2})+s_2(v-(s_1(v\bmod \W_{-2})\bmod \mathscr{I}^2)).$$

This is $\varphi_K$-equivariant because the same is true for $s_1, s_2$, and is a section by direct computation.  This completes the proof.
\end{proof}
% !TEX root = explicit_main.tex

\newcommand\ind{\mathrm{ind}}
\newcommand\hw{\mathrm{h.w.}}
\newcommand\Conil{C}
\newcommand\fin{\mathrm{fin}}
\newcommand\Tann{\mathcal T}
\newcommand\fiber{\ambrit{fiber}{fibre}\xspace}
\newcommand\can{\mathrm{can}}
\newcommand\LLambda{\mathbf{\Lambda}}
\newcommand{\MF}{\mathbf{MF}}
\newcommand{\adm}{\mathrm{w.a.}}
\newcommand\F{F}

\section{Structure of local Bloch--Kato Selmer schemes}\label{s:geometry}

We now turn to our application of these results to the study of local Bloch--Kato Selmer schemes. As before, let $K$ be a finite extension of $\BQ_p$, and let $U$ be a $\BQ_p$-pro-unipotent group endowed with a continuous action of $G_K$ (see Remark~\ref{rmk:topologies}). We assume moreover that $U$ is de Rham. We have in mind that $U$ is the $\BQ_p$-pro-unipotent \'etale fundamental group of a smooth geometrically connected variety, but will not assume this in what follows.

One can associate to $U$ a \emph{continuous Galois cohomology presheaf} $\HH^1(G_K,U)$ of pointed sets on the category $\Aff_{\BQ_p}$ of affine $\BQ_p$-schemes, namely the presheaf whose sections over some $\Spec(\Lambda)$ is the non-abelian Galois cohomology set $\HH^1(G_K,U(\Lambda))$. The \emph{local Bloch--Kato Selmer presheaves} (cf.\ \cite[\S2]{minhyong:siegel} \& \cite[Definition~1.2.1]{motivic_anabelian_heights}) are three sub-presheaves
\begin{equation}\label{eq:BK_presheaves}
\HH^1_e(G_K,U)\subseteq\HH^1_f(G_K,U)\subseteq\HH^1_g(G_K,U)\subseteq\HH^1(G_K,U)
\end{equation}
whose sections over some $\Spec(\Lambda)$ are given by
\begin{align*}
\HH^1_e(G_K,U)(\Lambda) &= \ker\left(\HH^1(G_K,U(\Lambda))\rightarrow\HH^1(G_K,U(\B_\cris^{\varphi=1}\otimes\Lambda))\right)\,, \\
\HH^1_f(G_K,U)(\Lambda) &= \ker\left(\HH^1(G_K,U(\Lambda))\rightarrow\HH^1(G_K,U(\B_\cris\otimes\Lambda))\right)\,, \\
\HH^1_g(G_K,U)(\Lambda) &= \ker\left(\HH^1(G_K,U(\Lambda))\rightarrow\HH^1(G_K,U(\B_\dR\otimes\Lambda))\right)\,.
\end{align*}

\begin{remark}[on topologies]\label{rmk:topologies}
Throughout this section, we will adhere to the standard conventions regarding topologies on $\BQ_p$-linear objects. A finite-dimensional $\BQ_p$-vector space will be endowed with its natural $p$-adic topology. A general $\BQ_p$-vector space, for instance $\OO(U)$ or a $\BQ_p$-algebra $\Lambda$ will be endowed with the inductive (co)limit topology over its finite-dimensional subspaces. A pro-finite-dimensional $\BQ_p$-vector space -- i.e.\ a pro-object in the category of finite-dimensional $\BQ_p$-vector spaces, such as $\Lie(U)$ -- will be endowed with the inverse limit topology over its finite-dimensional quotients. More generally, if $V=\varprojlim_i(V_i)$ is a pro-finite-dimensional $\BQ_p$-vector space, with each $V_i$ finite-dimensional, then for each $\BQ_p$-algebra $\Lambda$ we endow the base change $V_\Lambda:=\varprojlim_i(\Lambda\otimes_{\BQ_p}V_i)$ with the inverse limit of the inductive limit topologies on each $\Lambda\otimes_{\BQ_p}V_i$.

If $U$ is a $\BQ_p$-pro-unipotent group and $\Lambda$ is a $\BQ_p$-algebra, then the logarithm map provides a bijection $U(\Lambda)\isoarrow\Lie(U)_\Lambda$, and we topologi\sz e the set $U(\Lambda)$ by declaring this map to be a homeomorphism. We say that an action of a pro-finite group $G$ on $U$ is \emph{continuous} just when $G$ acts continuously on $U(\Lambda)$ for all $\BQ_p$-algebras $\Lambda$ (equivalently, $G$ acts continuously on $\Lie(U)$ or $\OO(U)$ \cite[Definition--Lemma~4.0.1]{motivic_anabelian_heights}). We say that $U$ is \emph{de Rham} just when $\Lie(U)$ is pro-de Rham (equivalently $\OO(U)$ is ind-de Rham \cite[Definition--Lemma~4.2.2]{motivic_anabelian_heights}).

Though the period rings $\B_\cris$, $\B_\st$ and $\B_\dR$ also carry topologies, none of our definitions or results here depend on these topologies.
\end{remark}

An argument of Kim establishes that the Bloch--Kato Selmer presheaf $\HH^1_f(G_K,U)$ is representable when $U$ is the maximal $n$-step unipotent quotient of the $\BQ_p$-pro-unipotent \'etale fundamental group of a smooth projective curve. The argument only uses a condition on the weights of $U$.

\begin{defn}\label{def:negative_weights}
We say that $U$ is \emph{mixed with only negative weights} just when $\OO(U)$ is endowed with an exhaustive filtration
\[
\BQ_p=\W_0\OO(U)\leq\W_1\OO(U)\leq\dots\leq\OO(U)
\]
by $G_K$-stable subspaces, stable under the Hopf algebra structure maps, such that $\gr^\W_i\OO(U)$ is pure of weight $i$ for all $i\geq0$. Equivalently, $\Lie(U)$ is endowed with a separated filtration
\[
\dots\leq\W_{-2}\Lie(U)\leq\W_{-1}\Lie(U)=\Lie(U)
\]
by $G_K$-stable pro-finite-dimensional subspaces, stable under the Lie algebra structure maps, such that $\gr^\W_{-i}\Lie(U)$ is pure of weight $-i$ for all $i>0$.
\end{defn}

\begin{lemma}[{\cite[Lemma~5]{minhyong:selmer}}]\label{lem:H^1_f}
Assume that $U$ is mixed with only negative weights. Then $\HH^1_f(G_K,U)$ is representable by an affine scheme over $\BQ_p$.
\end{lemma}

In fact, \cite[Proposition~1.4]{kim:tangential} gives a precise description of the representing scheme. Let $\D_\dR(U)$ be the pro-unipotent group over $K$ representing the presheaf
\[
\Spec(\Lambda)\mapsto\D_\dR(U)(\Lambda):= U(\B_\dR\otimes_K\Lambda)^{G_K}
\]
where $\B_\dR$ is the de Rham period ring (see \cite[Lemma~4.2.1]{motivic_anabelian_heights}). The argument of Kim establishes that in the setup of Lemma~\ref{lem:H^1_f} there is an isomorphism
\[
\HH^1_f(G_K,U)\cong\Res^K_{\BQ_p}\left(\D_\dR(U)/\F^0\right)
\]
of presheaves on $\Aff_{\BQ_p}$ (where the right-hand side denotes the presheaf quotient of $\D_\dR(U)$ by the right-multiplication action of the subgroup $\F^0\D_\dR(U)$ corresponding to the $0$th step of the Hodge filtration on $\Lie(U)$). In particular, the representing variety is an affine space: if one chooses a splitting $\D_\dR(\Lie(U))=V\oplus\F^0$ of the Hodge filtration on the Lie algebra of $U$, then there is an isomorphism
\[
\Res^K_{\BQ_p}V \cong \Res^K_{\BQ_p}\left(\D_\dR(U)/\F^0\right)
\]
of presheaves on $\Aff_{\BQ_p}$, where by abuse of notation we also denote by $V$ the associated affine space $\Spec(\Lambda)\mapsto V_\Lambda$.

Our aim in this section is to extend this to descriptions of all three Bloch--Kato Selmer presheaves, and in particular to show that under the same assumptions on the weights, they are all also represented by affine spaces. In fact, by imitating the arguments in \cite{alex-netan} we will obtain descriptions of the Bloch--Kato Selmer presheaves as (the affine spaces underlying) vector spaces; these descriptions are all canonical, up to the above choice of splitting of the Hodge filtration.

\begin{theorem}\label{thm:representability}
Let $U$ be a $\W$-filtered de Rham representation of $G_K$ on a finitely generated pro-unipotent group over $\BQ_p$, which is mixed with negative weights. Then there are \emph{canonical} natural isomorphisms
\begin{align*}
\HH^1_e(G_K,U) &\cong \Res^K_{\BQ_p}\left(\D_\dR(U)/\F^0\right) \\
\HH^1_f(G_K,U) &\cong \Res^K_{\BQ_p}\left(\D_\dR(U)/\F^0\right) \\
\HH^1_g(G_K,U) &\cong \Res^K_{\BQ_p}\left(\D_\dR(U)/\F^0\right)\times \V(U)^{\varphi=1}
\end{align*}
of presheaves for a certain $\varphi$-module $\V(U)$ functorially assigned to $U$ (for a precise description, see below). These descriptions are compatible with the inclusions~\eqref{eq:BK_presheaves}.

In particular all three presheaves are representable by affine spaces, and the dimension of these spaces is given by
\begin{align*}
\dim_{\BQ_p}\HH^1_e(G_K,U) &= [K:\BQ_p]\sum_{i>0}\left(\dim_K\D_\dR(\gr^\W_{-i}U)-\dim_K\F^0\D_\dR(\gr^\W_{-i}U)\right) \\
\dim_{\BQ_p}\HH^1_f(G_K,U) &= [K:\BQ_p]\sum_{i>0}\left(\dim_K\D_\dR(\gr^\W_{-i}U)-\dim_K\F^0\D_\dR(\gr^\W_{-i}U)\right) \\
\dim_{\BQ_p}\HH^1_g(G_K,U) &= [K:\BQ_p]\sum_{i>0}\left(\dim_K\D_\dR(\gr^\W_{-i}U)-\dim_K\F^0\D_\dR(\gr^\W_{-i}U)\right) \\
 &\:\:+\sum_{i>0}\dim_{\BQ_p}\D_\cris^{\varphi=1}((\gr^\W_{-i}U)^\dual(1))
\end{align*}
%\[
%\dim_{\BQ_p}\HH^1_*(G_K,U)=\sum_{i>0}\dim_{\BQ_p}\HH^1_*(G_K,\gr^\W_{-i}U)
%\]
where the right-hand side is the sum of the dimensions of the Bloch--Kato Selmer groups. If $*\in\{e,f\}$, or if $*=g$ and $U$ is Frobenius-semisimple, the same holds for the descending central series in place of the weight filtration.
\end{theorem}

\begin{example}
Suppose that $X/K$ is a smooth projective curve of genus $g$ with semistable reduction, and that all irreducible components of the geometric special fibre of the minimal regular model of $X$ are defined over the residue field $k$. Let $Y=X\setminus\{x\}$ for a point $x\in X(K)$ and let $U_n/\BQ_p$ denote the maximal $n$-step unipotent quotient of the $\BQ_p$-pro-unipotent \'etale fundamental group of $Y_{\overline K}$ (at a basepoint $b\in Y(K)$). Then we have
%\begin{align*}
%\dim_{\BQ_p}\HH^1_e(G_K,U_n) &= [K:\BQ_p]\cdot\sum_{1\leq i\leq n}\sum_{d\mid i}\frac{\mu(d)}i((2g)^{i/d}-g^{i/d}) \\
%\dim_{\BQ_p}\HH^1_f(G_K,U_n) &= [K:\BQ_p]\cdot\sum_{1\leq i\leq n}\sum_{d\mid i}\frac{\mu(d)}i((2g)^{i/d}-g^{i/d}) \\
%\dim_{\BQ_p}\HH^1_g(G_K,U_n) &= [K:\BQ_p]\cdot\sum_{1\leq i\leq n}\sum_{d\mid i}\frac{\mu(d)}i((2g)^{i/d}-g^{i/d}) \\
% &\:\:+ \frac1{g_0-1}\Bigl(g_0^{n+1}-g_0 + \\ 
% & \hspace{1.6cm} + \frac12\sum_j\nu_j^2\cdot(g_0^{n-1}-1) - 2g_1\cdot(g_0^{\lfloor n/2\rfloor}-1)\Bigr) \,,
%\end{align*}
\begin{align*}
\dim_{\BQ_p}\HH^1_g(G_K,U_n) &= [K:\BQ_p]\cdot\bigl(L_{\leq n}(2g)-L_{\leq n}(g)\bigr) - L_{\leq n}(g_0) + \frac{g_0^{n+1}-g_0}{g_0-1} \\
 &\:\:+ \frac{(n-1)g_0^n-ng_0^{n-1}+1}{2(g_0-1)^2}\cdot\sum_\lambda\nu_\lambda^2 - \frac{g_0^{\lfloor n/2\rfloor}-1}{2(g_0-1)}\cdot(\nu_{\sqrt q}+\nu_{-\sqrt q}) \,, \\
\end{align*}
where $g_0$ is the genus of the reduction graph of $X$, $\nu_\lambda$ are the multiplicities of the weight $1$ eigenvalues of $\varphi_K$ acting on $\HH^1_\et(X_{\overline K},\BQ_p)$, and $L_{\leq n}(T) := \sum_{1\leq i\leq n}\sum_{d\mid i}\frac{\mu(d)}iT^{i/d}$ is the summed necklace polynomial (the number of Lyndon words of length $\leq n$ in an alphabet of $T$ letters).
\begin{proof}[Proof (sketch)]
We will calculate the $\BQ_p$-dimension of $\D_\cris^{\varphi=1}((\gr^\W_{-i}U_n)^\dual(1))$ (for $i\leq n$), leaving the remainder of the calculation to the reader. This $\BQ_p$-dimension is equal to the $K_0$-dimension of the subspace of $\D_\st(\gr^\W_{-i}U_n^\dual)$ on which $\varphi_K$ acts via $q$ and $N$ acts by $0$. Since $U_n$ is a free $n$-step unipotent group, $\D_\st(\gr^\W_{-i}U_n^\dual)$ has a basis parametrised by Lyndon words in a basis of $\D_\st(\HH^1_\et(X_{\overline K},\BQ_p))$.

By semisimplicity, we may pick a basis of $\overline K\otimes_{K_0}\D_\st(\HH^1_\et(X_{\overline K},\BQ_p))$ consisting of $\varphi_K$-eigenvectors. Exactly $g_0$ of the corresponding eigenvalues are equal to $1$ and $g_0$ are equal to $q$; all the remaining eigenvalues are $q$-Weil numbers of weight $1$. The monodromy operator $N$ maps the $q$ eigenspace isomorphically onto the $1$ eigenspace, and acts as $0$ on all other eigenspaces. The corresponding Lyndon basis of $\D_\st(\gr^\W_{-i}U_n^\dual)$ is also a basis of $\varphi_K$-eigenvectors, with the eigenvalue of (the basis element corresponding to) a Lyndon word $w$ being the product of the eigenvalues of its letters.

It follows from this description that the monodromy operator on $\gr^\W_{-i}U_n^\dual$ maps the $q$ eigenspace surjectively onto the $1$ eigenspace, and so the desired dimension is equal to the number of Lyndon words of length $i$ and eigenvalue $q$, minus the number of Lyndon words of length $i$ and eigenvalue $1$. This latter quantity is simply the number of Lyndon words in the $g_0$ vectors of eigenvalue $1$, and hence equal to $\sum_{d\mid i}\frac{\mu(d)}ig_0^{i/d}$ \cite[Theorem~7.1]{reutenauer}. Summed over $1\leq i\leq n$, this yields the term $-L_{\leq n}(g_0)$ in the claimed formula.

Now there are two types of Lyndon words $w$ of eigenvalue $q$: the letters of $w$ with eigenvalue not equal to $1$ are either a single eigenvector with eigenvalue $q$, or two eigenvectors with eigenvalues $\lambda$, $q/\lambda$ with $\lambda$ of weight $1$. To count words of the former type, suppose that $x$ is an eigenvector with eigenvalue $q$. There are $g_0^{i-1}$ Lyndon words of length $i$ containing $x$ and $i-1$ vectors of eigenvalue $0$ by \cite[Theorem~7.1(7.1.2)]{reutenauer}. Summed over $x$ and over $1\leq i\leq n$, this yields the term $\frac{g_0^{n+1}-g_0}{g_0-1}$ in the claimed formula.

To count words of the latter type, suppose that $x$ and $y$ are eigenvectors with eigenvalues $\lambda$ and $q/\lambda$ respectively. Using \cite[Theorem~7.1(7.1.2)]{reutenauer} again, we find that if $x\neq y$, then the number of Lyndon words containing $x$, $y$ and $i-2$ eigenvectors of eigenvalue $1$ is $(i-1)g_0^{i-2}$. If $x=y$, then the number of Lyndon words containing two copies of $x$, $y$ and $i-2$ eigenvectors of eigenvalue $1$ is $\frac{i-1}2g_0^{i-2}$ if $i$ is odd, and $\frac{i-1}2g_0^{i-2}-\frac12g_0^{i/2-1}$ if $i$ is even. Summed over $x,y$ and over $1\leq i\leq n$, this yields the quantity in the final line of the claimed formula. Here, we are using that $\nu_\lambda=\nu_{q/\lambda}$ by Poincar\'e duality, so that there are $\frac12\left(\sum_\lambda\nu_\lambda^2-\nu_{\sqrt q}-\nu_{-\sqrt q}\right)$ unordered pairs with $x\neq y$, and $\nu_{\sqrt q}+\nu_{-\sqrt q}$ pairs with $x=y$.
\end{proof}
\end{example}

\subsection{Torsors in Tannakian categories}

We will prove Theorem~\ref{thm:representability} via an alternative interpretation of the presheaves $\HH^1_*(G_K,U)$ as classifying spaces of certain $G_K$-equivariant torsors. For our purposes, it will be convenient to recall the definition of torsors in slightly greater generality.

Let $\Tann$ be a Tannakian category -- not necessarily neutral -- over a characteristic $0$ field $\Fd$. As explained in \cite[\S5.2--5.4]{deligne:p1}, one can make sense of many of the basic notions of affine algebraic geometry inside $\Tann$, defining algebras in $\ind{-}\Tann$, affine schemes in $\Tann$, and so on, all functorial with respect to faithfully exact $\otimes$-functors. We say that an affine group scheme $U$ in $\Tann$ is \emph{pro-unipotent} just when $\omega(\Tann)$ is pro-unipotent for some \fiber functor $\omega\colon\Tann\rightarrow\Mod_{\Fd'}^\fin$ valued in finite-dimensional $\Fd'$-vector spaces for some extension field $\Fd'\supseteq\Fd$. Since any two \fiber functors become isomorphic over a common field extension, this notion is independent of $\omega$.

A \emph{torsor} under an affine group scheme $U$ in $\Tann$ over an affine scheme $\Spec(\LLambda)$ in $\Tann$ is a faithfully flat $\Spec(\LLambda)$-scheme $T\rightarrow\Spec(\LLambda)$ endowed with a \ambrit{fiberwise}{fibrewise} right action of $U$ such that the induced map
\[
T\times_{\Spec(\1)}U \isoarrow T\times_{\Spec(\LLambda)}T
\]
is an isomorphism. Here, faithfully flat means\footnote{There is an alternative internal definition of faithful flatness in \cite[\S5.3]{deligne:p1}. Although these presumably define the same notion, we will neither need nor prove this in what follows.} that the map $\omega(T)\rightarrow\omega(\Spec(\LLambda))$ is faithfully flat for some \fiber functor $\omega$ -- as above this is independent of $\omega$. When $U$ is pro-unipotent, the condition that $\omega(T)\rightarrow\omega(\Spec(\LLambda))$ is faithfully flat is equivalent to it being split, i.e.\ $\omega(T)(\omega(\LLambda))$ being non-empty.

There is a functor $\Lambda\mapsto\Lambda\otimes\1$ from $\Fd$-algebras to algebras in $\Tann$ \cite[\S5.6]{deligne:p1}, and so we define the \emph{first cohomology presheaf} of an affine group-scheme $U$ in $\Tann$ to be the presheaf on $\Aff_\Fd$ given by
\[
\HH^1(\Tann,U)\colon\Spec(\Lambda)\mapsto\{\text{$U$-torsors over $\Spec(\Lambda\otimes\1)$}\}/\text{iso}\,.
\]
This is in fact a presheaf of pointed sets, with basepoint corresponding to the class of the trivial torsor $U$ over $\Spec(\1)$.

%\begin{example}
%Let $\Tann=\Mod_\Fd^\fin$ be the category of finite-dimensional vector spaces over $\Fd$. Then the category of affine schemes in $\Tann$ is just the usual category of affine schemes, and the notions of pro-unipotent group, torsor etc.\ agree with the usual definitions. For every $U$-torsor $T$ over $\Spec(\Lambda)$, the set $T(\Lambda)$ is non-empty: when $U=\GG_a$ this follows from the vanishing of the fpqc coherent cohomology group $\HH^1(\Spec(\Lambda),\OO_{\Spec(\Lambda)})$, and in general via a careful transfinite induction. In other words, $\HH^1(\Tann,U)=1$ is the constant presheaf.
%\end{example}

\begin{example}\label{ex:equivariant_torsors}
Let $G$ be a profinite group, and $\Tann$ the category of continuous representations of $G$ on finite-dimensional $\BQ_\ell$-vector spaces. Then $\ind{-}\Tann$ is the category of continuous representations of $G$ on general $\BQ_\ell$-vector spaces, where the topology on a $\BQ_\ell$-vector space is as in Remark~\ref{rmk:topologies}. The category of affine schemes in $\Tann$ is equivalent to the category of affine $\BQ_\ell$-schemes $X$ endowed with an action of $G$ which is continuous in the sense that the action map $G\times X(\Lambda)\rightarrow X(\Lambda)$ is continuous for every $\BQ_\ell$-algebra $\Lambda$. Here, the topology on $X(\Lambda)$ is the subspace topology induced from any (possibly infinite-dimensional) affine embedding $X(\Lambda)\hookrightarrow\Spec(\Sym^\bullet(V))(\Lambda)=V^\dual_\Lambda$, where $V^\dual_\Lambda$ is topologi\sz ed as in Remark~\ref{rmk:topologies}.

A pro-unipotent group in $\Tann$ is a $\BQ_\ell$-pro-unipotent group endowed with a continuous action of $G$. A $U$-torsor $T$ over $\Spec(\Lambda\otimes\1)$ is determined up to isomorphism by the $U(\Lambda)$-valued cohomology class of the cocycle $\sigma\mapsto\gamma^{-1}\sigma(\gamma)$ for $\gamma\in T(\Lambda)$, and hence there is an isomorphism
\[
\HH^1(\Tann,U) \cong \HH^1(G,U)
\]
of presheaves on $\Aff_{\BQ_\ell}$, where $\HH^1(G,U)$ denotes the continuous Galois cohomology presheaf.
\end{example}

\begin{example}
Let $\Tann=\Mod_{(\varphi,N,G_K)}$ be the $\BQ_p$-linear Tannakian category of discrete $(\varphi,N,G_K)$-modules. There are two natural \fiber functors on $\Tann$: one given by taking the underlying $\BQ_p^\nr$-vector space, and one defined over $K$ given by taking the $G_K$-invariants in the base change to $\overline K$ \cite[4.3.1]{fontaine:semi-stables}. We denote these by $(-)_{\BQ_p^\nr}$ and $(-)_K$, respectively. If $D$ is the $(\varphi,N,G_K)$-module associated to a de Rham representation $V$, then $D_{\BQ_p^\nr}=\D_\pst(V)$ and $D_K=\D_\dR(V)$.

For an affine scheme $T$ in $\Tann$, $T_{\BQ_p^\nr}$ carries various extra structures, including a crystalline Frobenius $\varphi\colon T_{\BQ_p^\nr}\isoarrow \sigma^*T_{\BQ_p^\nr}$, where $\sigma^*$ denotes pullback along the arithmetic Frobenius in $G_{\BQ_p^\nr/\BQ_p}$, as well as a vector field determined by the monodromy operator. In particular, the crystalline Frobenius acts on Hom-sets between affine schemes in $\Tann$.
\end{example}

In what follows, we will often make use of the following well-known lemma, which we have already seen implicitly in Definition~\ref{weight-defn}.

\begin{lemma}[Transport of filtrations]\label{lem:transport}
Let $U$ be a pro-unipotent group in $\Tann$ endowed with an exhaustive and increasing filtration
\[
\1=\W_0\OO(U)\leq\W_1\OO(U)\leq\dots
\]
by objects in $\ind{-}\Tann$, compatible with the Hopf algebra structure. Then for every $U$-torsor $T$ over $\Spec(\LLambda)$ there is a unique exhaustive filtration $\W_\bullet\OO(T)$ of $\OO(T)$ by objects in $\ind{-}\Tann$ compatible with the $\LLambda$-algebra structure and the coaction map $\OO(T)\rightarrow\OO(T)\otimes\OO(U)$. There is a unique isomorphism $\gr^\W_\bullet\OO(T)\isoarrow\LLambda\otimes\gr^\W_\bullet\OO(U)$ of graded $\LLambda$-algebras in $\ind{-}\Tann$ compatible with the graded coaction map.
\begin{proof}[Proof (sketch)]
When $\Tann=\Mod_\Fd^\fin$, this is well-known: one chooses an element $\gamma\in T(\Lambda)$, which provides an isomorphism $T\isoarrow\Spec(\LLambda)\times_{\Spec(\1)}U$ of $U$-torsors over $\Spec(\LLambda)$, and the $\W$-filtration on $\OO(T)$ is defined so as to make $\OO(T)\isoarrow\LLambda\otimes\OO(U)$ a $\W$-filtered isomorphism. This filtration is independent of $\gamma$, as is the induced $\W$-graded isomorphism $\gr^\W_\bullet\OO(T)\isoarrow\LLambda\otimes\gr^\W_\bullet\OO(U)$.

Both the $\W$-filtration and the induced graded isomorphism admit intrinsic descriptions. For the $\W$-filtration, $\W_n\OO(T)$ is the preimage of $\OO(T)\otimes\W_n\OO(U)$ under the coaction map, while the graded isomorphism is the composite $\gr^\W_\bullet\OO(T)\rightarrow\gr^\W_\bullet\OO(T)\otimes\gr^\W_\bullet\OO(U)\rightarrow\LLambda\otimes\gr^\W_\bullet\OO(U)$, where the first map is the graded coaction and the second is the inverse of the unit isomorphism $\LLambda\isoarrow\gr^\W_0\OO(T)$, tensored with $\gr^\W_\bullet\OO(U)$.

This latter definition of $\W_\bullet\OO(T)$ makes sense in any Tannakian category $\Tann$; that this gives the unique filtration with the claimed properties can be checked after applying a \fiber functor, where we use the previous case (over a field extension). This in particular ensures that the unit map $\LLambda\isoarrow\gr^\W_0\OO(T)$ is an isomorphism, so that the definition of the graded isomorphism also makes sense in $\Tann$. Again applying a \fiber functor proves that it is the unique isomorphism with the claimed properties.
\end{proof}
\end{lemma}

%\begin{corollary}\label{cor:cohomology_wrt_subcat}
%Suppose that $\Tann'\subseteq\Tann$ is a full Tannakian subcategory closed under extensions, and that $U$ is a pro-unipotent group in $\Tann'$. Then the induced map
%\[
%\HH^1(\Tann',U) \rightarrow \HH^1(\Tann,U)
%\]
%is an isomorphism of presheaves on $\Aff_\Fd$.
%\begin{proof}
%The key point to verify is that any $U$-torsor over $\LLambda:=\Lambda\otimes\1$ in $\Tann$ is in fact in $\Tann'$. Let $\W_\bullet$ denote the conilpotency filtration on $\OO(U)$, as in ???. Lemma~\ref{lem:transport} shows that $\gr^\W_\bullet\OO(P)\cong\LLambda\otimes\gr^\W_\bullet\OO(U)$ in $\Tann$, and hence $\gr^\W_\bullet\OO(P)$ lies in $\ind{-}\Tann'$. Since $\Tann'$ is closed under extensions we obtain then that $\OO(P)$ lies in $\ind{-}\Tann'$. The remainder of the proof is routine.
%\end{proof}
%\end{corollary}

\subsubsection{Torsors in mixed Weil--Deligne representations}

The main example we will compute is the case when $\Tann$ is the category of mixed Weil--Deligne representations over a characteristic $0$ field $\Fd$. Throughout this subsection, we fix a pro-unipotent group $U$ in $\Rep_\Fd^\mix(\WeilD_K)$ (equivalently, an $\Fd$-pro-unipotent group $U$ with an action of the Tannaka group $\G_\mix$ from \S\ref{ss:tannaka}), and assume that $U$ has only negative weights in the sense that $\W_0\OO(U)=\BQ_\ell$ (equivalently $\W_{-1}\Lie(U)=\Lie(U)$). We will describe the cohomology of $U$ explicitly in terms of the following vector space (which was denoted by $\Lie(U)^\can$ in \cite[Definition~2.2.2]{alex-netan} in the context of $\ell$-adic Galois representations).

\begin{defn}
The Lie algebra $\Lie(U)$, being the space of $\Fd$-valued derivations on $\OO(U)$, is naturally a pro-object of $\Rep_\Fd^\mix(\WeilD_K)$. We define
\[
\V(U) := \Lie(U)(-1)^{\Weil_K,X}
\]
to be the (pro-finite dimensional) subspace of $\Lie(U)(-1)$ fixed by $\Weil_K$ and $X=\begin{pmatrix}1&1\\0&1\end{pmatrix}\in\SL_2(\Fd)$ (see \S\ref{ss:tannaka} for the definition of the $\SL_2$-action). In other words, $\V(U)$ consists of those elements $v\in\Lie(U)$ such that $\rho(w)(x)=p^{-v(w)}\cdot v$ for all $w\in\Weil_K$ and $N^r(v)\in\W_{-r-2}\Lie(U)$ for all $r\geq0$.

Since the weight torus is central in $\SL_2\rtimes\Weil_K^\mix$, it acts on $\V(U)$ and endows it with a product grading\footnote{The grading we actually use here is the grading on $\Lie(U)$, which we identify with $\Lie(U)(-1)$ via the triviali\sz ation of $\Fd(1)$. In other words, we have shifted the grading on $\Lie(U)(-1)$ by two. This is because we will see that $\gr^\W_{-r}\V$ is related to the operator $\delta_r$ from Definition~\ref{def:mixing}, which is $\W$-graded of degree $-r$.}
\[
\V(U) = \prod_{r>0}\gr^\W_{-r}\V(U)\,.
\]
Note that $\gr^\W_{-1}\V(U)=0$.
\end{defn}

\begin{lemma}\label{lem:mixed_torsors}
There is a canonical isomorphism
\[
\HH^1(\Rep_\Fd^\mix(\WeilD_K),U) \cong \V(U)
\]
of presheaves on $\Aff_\Fd$, where we identify the pro-finite-dimensional vector space $\V(U)$ with its corresponding affine space in the usual way.
\end{lemma}

The key to this theorem is the following version of the canonical paths from Definition~\ref{defn:can-paths} over a base.

\begin{prop}\label{prop:canonical_path_over_base}
For every $U$-torsor $T$ over $\Spec(\Lambda)$, there is a unique element $p_T\in T(\Lambda)$ fixed under the action of the weight torus. This element is in fact invariant under the action of $\G_\pure$.
\begin{proof}
The action of the weight torus~$\GG_m$ on~$T$ corresponds to the canonical splitting of the weight filtration on~$\OO(T)$, and the inclusion $T^{\GG_m}\hookrightarrow T$ corresponds to the projection $\OO(T)\to\gr^\W_0\OO(T)$ of the canonical splitting. According to Lemma~\ref{lem:transport}, we have $\gr^\W_0\OO(T)=\Lambda$, and hence $T^{\GG_m}\cong\Spec(\Lambda)$ consists of a single~$\Lambda$-point $p_T\in T(\Lambda)$. Since $\GG_m$ is central in $\G_\pure$, it follows that the action of $\G_\pure$ preserves $T^{\GG_m}$, so must fix~$p_T$.
\end{proof}
\end{prop}

\begin{proof}[Proof of Lemma~\ref{lem:mixed_torsors}]
Recall from Lemma~\ref{lem:explicit_tannaka} that the Tannaka group $\G_\mix$ of $\Rep_\Fd^\mix(\WeilD_K)$ is canonically isomorphic to $\mathcal U\rtimes\G_\pure$, where~$\G_\pure$ is the Tannaka group of pure Weil--Deligne representations and~$\mathcal U$ is a pro-unipotent group containing certain elements $\exp(\delta_r)$ for~$r>0$. We define the map $\alpha\colon\HH^1(\Rep_\Fd^\mix(\WeilD_K),U)\rightarrow\V(U)=\prod_{r>0}\gr^\W_{-r}\Lie(U)(-1)^{\Weil_K,X}$ as follows. Given a $U$-torsor $T$ over some $\Spec(\Lambda)$ and an integer $r>0$, by differentiating the action of the copy of $\GG_a$ in $\mathcal U\leq\G_\mix$ spanned by $\exp(\delta_r)$, we obtain a relative vector field $D_r$ on $T$. Evaluating this vector field at the canonical path $p_T$ from Proposition~\ref{prop:canonical_path_over_base} gives a relative tangent vector based at $p_T$, and hence $p_T^{-1}D_r(p_T)$ is an element of $\Lie(U)(\Lambda)$. The commutation conditions for $\delta_r$ (see Remark~\ref{rmk:X}) ensure that in fact $p_T^{-1}D_r(p_T)\in\gr^\W_{-r}\Lie(U)(-1)^{\Weil_K,X}$, so we define the map $\alpha$ by $\alpha_{\Spec(\Lambda)}([T]) := (p_T^{-1}D_r(p_T))_{r>0}$.

Next we show that each $\alpha_{\Spec(\Lambda)}$ is injective. Suppose we are given two torsors $T_1$, $T_2$ over $\Spec(\Lambda)$ such that $p_{T_1}^{-1}D_r(p_{T_1})=p_{T_2}^{-1}D_r(p_{T_2})$ for all $r>0$. There is a unique isomorphism $\beta\colon T_1\isoarrow T_2$ of $U$-torsors taking $p_{T_1}$ to $p_{T_2}$, and this is automatically $\G_\pure$-equivariant. To show that $\beta$ is equivariant for the action of the copy of $\GG_a$ in $\mathcal U$ generated by $\exp(\delta_r)$, it suffices to show that it is compatible with the relative vector fields $D_r$. But since the action map $T_i\times U\rightarrow T_i$ is $\GG_a$-equivariant, it is automatically compatible with the vector fields $D_r$ on $T_i$ and $U$, so the vector field $D_r$ on $T_i$ is determined by its value at any given $\Lambda$-point of $T_i$. Since by assumption $\beta$ preserves $D_r$ at $p_{T_1}$, it thus preserves $D_r$ everywhere, and so is $\GG_a$-equivariant. Since~$\G_\mix$ is generated by~$\G_\pure$ and the elements $\exp(\delta_r)$ for~$r>0$, the isomorphism~$\beta$ is $\G_\mix$-equivariant, so $[T_1]=[T_2]$.

Finally, we show that each $\alpha_{\Spec(\Lambda)}$ is surjective, by constructing an explicit inverse. For a tuple $\underline v=(v_r)_{r>0}\in\prod_{r>0}\gr^\W_{-r}\Lie(U)(-1)_\Lambda^{\Weil_K,X}$, we let $U_{\underline v}$ denote the $\G_\mix$-equivariant $U$-torsor over $\Spec(\Lambda)$ constructed as follows. We take $U_{\underline v}=\Spec(\Lambda)\times_{Spec(\Fd)}U$ (as a right $U$-torsor) with the same action of $\G_\pure$. For $r>0$, we define an action of $\GG_a$ on $U_{\underline v}$ by $\rho_r(t)(u)=\exp(tv_r)\cdot u$. The choice of $v_r$ ensures that $\rho_r(t)$ commutes with the action of $X$, and satisfies the commutation relations
\begin{align*}
\rho(w)\circ\rho_r(t)\circ\rho(w) &= \rho_r(p^{-v(w)}t) \\
\rho(\lambda)\circ\rho_r(t)\circ\rho(\lambda^{-1}) &= \rho_r(\lambda^{-r}t)
\end{align*}
against elements $w\in\Weil_K$ in the Weil group and $\lambda\in\GG_m$ in the weight torus. It is easy to check that these actions extend uniquely to an action of $\mathcal U$ on $U_{\underline v}$ equivariant for the action of $U$ and $\G_\pure$, and hence endow it with the structure of a $\G_\mix$-equivariant $U$-torsor. It is then easy to check that $\alpha([U_{\underline v}]) = \underline v$, so that $\alpha$ is surjective, as claimed.
\end{proof}

Before we proceed to the proof of Theorem~\ref{thm:representability}, we note that we already have proved its $\ell$-adic counterpart, re-proving \cite[Theorem~2.2.4]{alex-netan} (in mildly more generality).

\begin{corollary}\label{cor:explicit_l-adic}
Let $\ell\neq p$ be a prime and let $U$ be a $\BQ_\ell$-pro-unipotent group endowed with a continuous action of $G_K$, and suppose that $U$ is mixed with only negative weights with respect to some filtration on $\OO(U)$. Then we have a canonical isomorphism
\[
\HH^1(G_K,U) \cong \V(U)
\]
of presheaves on $\Aff_{\BQ_\ell}$. In particular, $\HH^1(G_K,U)$ is representable by an affine space.
\begin{proof}
The filtration on $\Lie(U)$ induces a corresponding filtration $\BQ_\ell=\W_0\OO(U)\leq\W_1\OO(U)\leq\dots$ making $\OO(U)$ into an ind-mixed representation of $G_K$. If $T$ is a $U$-torsor over $\Spec(\Lambda)$, then by Lemma~\ref{lem:transport} there is a unique $G_K$-invariant $\W$-filtration on $\OO(P)$ compatible with all the appropriate structures, which makes $\OO(P)$ into an ind-mixed representation of $G_K$. We see from this construction and Example~\ref{ex:equivariant_torsors} that $\HH^1(G_K,U)$ is isomorphic to $\HH^1(\Rep_{\BQ_\ell}^\mix(G_K),U)$. It follows from the full-faithfulness of the functor
\[
\Rep_{\BQ_\ell}^\mix(G_K)\rightarrow \Rep_{\BQ_\ell}^\mix(\WeilD_K)
\]
of Example~\ref{ex:l_not_p} that the induced map
\[
\HH^1(\Rep_{\BQ_\ell}^\mix(G_K),U)\rightarrow \HH^1(\Rep_{\BQ_\ell}^\mix(\WeilD_K),U)
\]
is injective; surjectivity follows from the fact that the image of this fully faithful functor is closed under extensions. We conclude via Lemma~\ref{lem:mixed_torsors}.
\end{proof}
\end{corollary}

\subsubsection{Proof of Theorem~\ref{thm:representability}}

We now come to the proof of Theorem~\ref{thm:representability}. We proceed in several steps.

\begin{prop}
$\HH^1_g(G_K,U)\subseteq\HH^1(G_K,U)$ consists of those $G_K$-equivariant $U$-torsors which are de Rham. Such a torsor $T$ is automatically mixed (i.e.\ a torsor under $U$ in $\Rep_{\BQ_p}^\mix(G_K)$) when $\OO(T)$ is endowed with the unique weight filtration compatible with the weight filtration on $\OO(U)$, as in Lemma~\ref{lem:transport}.
\begin{proof}
For the first part, if $\OO(T)$ is ind-de Rham then we have a $G_K$-equivariant $\B_\dR\otimes\Lambda$-algebra isomorphism $\B_\dR\otimes_{\BQ_p}\OO(T)\isoarrow\B_\dR\otimes_K\D_\dR(\OO(T))$. Since $\D_\dR$ is a \fiber functor on the category of de Rham representations, $\D_\dR(T):=\Spec(\D_\dR(\OO(T)))$ is automatically a $\D_\dR(U)$-torsor over $\Spec(K\otimes\Lambda)$, and hence $\D_\dR(T)(K\otimes\Lambda)\neq\emptyset$. In particular, $T(\B_\dR\otimes\Lambda)^{G_K}=\D_\dR(\B_\dR\otimes\Lambda)^{G_K}\neq\emptyset$, and hence the class of $T$ lies in $\HH^1_g$.

In the other direction, let $\Conil_\bullet$ denote the conilpotency filtration on $\OO(U)$, and also the induced filtration on $\OO(T)$ from Lemma~\ref{lem:transport}, so that $\gr^\Conil_\bullet\OO(P)\cong\Lambda\otimes_{\BQ_p}\gr^\Conil_\bullet\OO(U)$ is ind-de Rham. When $T$ lies in $\HH^1_g$, we have that $T(\B_\dR\otimes_{\BQ_p}\Lambda)^{G_K}\neq\emptyset$, from which we deduce that there exists a $G_K$-equivariant and $\Conil$-filtered isomorphism $\B_\dR\otimes_{\BQ_p}\OO(T)\simeq\B_\dR\otimes_{\BQ_p}\Lambda\otimes_{\BQ_p}\OO(U)$. Since $\OO(U)$ is ind-de Rham, the conilpotency filtration on $\B_\dR\otimes_{\BQ_p}\OO(U)$ splits $\B_\dR$-linearly and $G_K$-equivariantly, and hence so too does the conilpotency filtration on $\B_\dR\otimes_{\BQ_p}\OO(T)$. This implies that $\OO(T)$ is ind-de Rham, as desired.

The second part follows by the same argument as in Corollary~\ref{cor:explicit_l-adic}.
\end{proof}
\end{prop}

\begin{prop}
Let $T$ be a de Rham $U$-torsor over $\Spec(\Lambda)$. Then:
\begin{enumerate}
	\item the canonical path $p_T\in\D_\pst(T)(\Lambda\otimes_{\BQ_p}\BQ_p^\nr)$ is $\varphi$- and $G_K$-invariant; and
	\item there exists a ``Hodge-filtered path'' $p_T^\dR\in\D_\dR(T)(\Lambda\otimes_{\BQ_p}K)$, such that the induced isomorphism $\OO(\D_\dR(T))\isoarrow\Lambda\otimes_{\BQ_p}\OO(\D_\dR(U))$ is a Hodge-filtered isomorphism. This choice of $p_T^\dR$ is unique up to the right-multiplication action of $\F^0\D_\dR(U)$ (the pro-unipotent subgroup whose Lie algebra is $\F^0\Lie(\D_\dR(U))$).
\end{enumerate}
In particular, $p_T\in\D_\dR(\Lambda\otimes_{\BQ_p}K)=T_{\BQ_p}(\Lambda\otimes_{\BQ_p}\overline K)^{G_K}$.
\begin{proof}
For the first part, $\varphi$-invariance of $p_T$ follows from the fact that the canonical splitting of the weight filtration on a mixed Weil--Deligne representation is $\varphi$-equivariant (Remark~\ref{rmk:varphi_actions}). Since it is also invariant under the Weil group action, it is thus invariant under the $G_K$-action.

The second part follows by the argument in \cite[Proposition~9.9]{motivic_anabelian_heights}.
\end{proof}
\end{prop}

Using this, we define a morphism $\alpha\colon\HH^1_g(G_K,U)\rightarrow\Res^K_{\BQ_p}\left(\D_\dR(U)/\F^0\right)\times\V(U)^{\varphi=1}$ by declaring that
\[
\alpha([T]) := (p_T^{-1}p_T^\dR,(p_T^{-1}D_r(p_T))_{r>0})\,,
\]
where $D_r$ is as in the proof of Lemma~\ref{lem:mixed_torsors}.

\begin{prop}\label{prop:alpha_iso}
$\alpha$ is an isomorphism.
\begin{proof}
Firstly, we show that $\alpha$ is injective as a morphism of presheaves. Suppose that $T_1$ and $T_2$ are de Rham $U$-torsors such that $\alpha([T_1])=\alpha([T_2])$. We have isomorphisms $\beta_\dR\colon\D_\dR(T_1)\isoarrow\D_\dR(T_2)$ and $\beta_\pst\colon\D_\pst(T_1)\isoarrow\D_\pst(T_2)$ (of $\D_\dR(U)$- and $\D_\pst(U)$-torsors, respectively) sending $p_{T_1}$ to $p_{T_2}$. Our assumptions that $\alpha([T_1])=\alpha([T_2])$ ensure firstly that $\beta_\dR$ is compatible with the Hodge filtration (on the affine rings) and secondly that $\beta_\pst$ is an isomorphism on the underlying mixed Weil--Deligne representations. Since $\beta_\pst$ is also $\varphi$-equivariant (as $p_{T_1}$ and $p_{T_2}$ are $\varphi$-invariant), it follows that $\beta_\pst$ is an isomorphism of $(\varphi,N,G_K)$-modules. Thus $\D_\pst(T_1)$ and $\D_\pst(T_2)$ are isomorphic as $\D_\pst(U)$-torsors in the category of weakly admissible Hodge-filtered discrete $(\varphi,N,G_K)$-modules. Since $\D_\pst\colon\Rep_{\BQ_p}^\dR(G_K)\isoarrow\MF^{\adm}_{(\varphi,N,G_K)}$ is an equivalence \cite[Th\'eor\`eme~A]{colmez-fontaine:construction_des_reps}\footnote{Strictly speaking, \cite[Th\'eor\`eme~A]{colmez-fontaine:construction_des_reps} only shows that weakly admissible $(\varphi,N)$-modules are admissible. However, it implies the same result for $(\varphi,N,G_K)$-modules by a straightforward argument. If~$D$ is a weakly admissible $(\varphi,N,G_K)$-module, then there is a finite Galois extension $L/K$ contained in $\Kbar$ such that $D=\Kbar_0\otimes_{L_0}D_0$ for some weakly admissible $(\varphi,N,G_{L|K})$-module~$D_0$. Since~$D_0$ is weakly admissible as a $(\varphi,N)$-module by \cite[Proposition~4.4.9]{fontaine:semi-stables}, it follows from the result of Colmez--Fontaine that $V_\st(D_0)$ is a semistable representation of~$G_L$. Hence $V_\pst(D)=V_\st(D_0)$ is a potentially semistable representation of~$G_K$, and $D\cong\D_\pst(V_\pst(D))\to D$ is admissible.}\cite[Th\'eor\`eme~5.6.7(v)]{fontaine:semi-stables}, it follows that $T_1$ and $T_2$ are $G_K$-equivariantly isomorphic as $U$-torsors.

To conclude, we show that $\alpha$ is surjective. Let $\Lambda$ be a $\BQ_p$-algebra and take some
\[
(u_\dR,(v_r)_{r>0})\in\D_\dR(U)(\Lambda\otimes_{\BQ_p}K)\times\prod_{r>0}\gr^\W_{-r}\Lie(U)(-1)_{\Lambda\otimes_{\BQ_p}\BQ_p^\nr}^{\Weil_K,X,\varphi=1}\,.
\]
We define a de Rham $U$-torsor $T_{u_\dR,\underline v}$ as follows. We set
\begin{align*}
\D_\dR(T_{u_\dR,\underline v}) &:= \Spec(\Lambda\otimes_{\BQ_p}K)\times_{\Spec(K)}\D_\dR(U)\,, \\
\D_\pst(T_{u_\dR,\underline v}) &:= \Spec(\Lambda\otimes_{\BQ_p}\BQ_p^\nr)\times_{\Spec(\BQ_p^\nr)}\D_\pst(U)\,,
\end{align*}
which are torsors under $\D_\dR(U)$ and $\D_\pst(U)$ respectively. We define a comparison isomorphism $c_{u_\dR}\colon\Spec(\overline K)\times_{\Spec(K)}\D_\dR(T_{u_\dR,\underline v})\isoarrow\Spec(\overline K)\times_{\Spec(\BQ_p^\nr)}\D_\pst(T_{u_\dR,\underline v})$ to be the map given by left-multiplication by $u_\dR$ on $\Spec(\Lambda\otimes_{\BQ_p}\overline K)\times_{\Spec(\Lambda\otimes_{\BQ_p}K)}\D_\dR(U)=\Spec(\Lambda\otimes_{\BQ_p}\overline K)\times_{\Spec(\Lambda\otimes_{\BQ_p}\BQ_p^\nr)}\D_\pst(U)$.

Now we endow $\D_\pst(T_{u_\dR,\underline v})$ with the semilinear crystalline Frobenius (on its affine ring) induced from the crystalline Frobenius on $\D_\pst(U)$. The proof of Lemma~\ref{lem:mixed_torsors} explains also how to use the elements $(v_r)_{r>0}$ to endow $\D_\pst(T_{u_\dR,\underline v})$ with the structure of a $\BQ_p^\nr$-linear mixed Weil--Deligne representation, whose underlying Weil group action is the same as that induced from $\D_\pst(U)$. It thus follows, by reversing the construction in Example~\ref{ex:l_is_p}, that this crystalline Frobenius and mixed Weil--Deligne representation underlie a mixed $(\varphi,N,G_K)$-module structure on $\D_\pst(T_{u_\dR,\underline v})$.

We also endow $\D_\dR(T_{u_\dR,\underline v})$ with the Hodge filtration arising from that on $\D_\dR(U)$. Via the comparison isomorphism $c_{u_\dR}$, this endows $\Spec(\Lambda\otimes_{\BQ_p}\overline K)\times_{\Spec(\Lambda\otimes_{\BQ_p}\BQ_p^\nr)}\D_\pst(T_{u_\dR,\underline v})$ with a $G_K$-invariant Hodge filtration. We claim that this Hodge filtration is weakly admissible.

To see this, by the argument of Lemma~\ref{lem:transport}, the $\W$-graded isomorphism
\[
\gr^\W_\bullet c_{u_\dR}^*\colon \gr^\W_\bullet\overline K\otimes_{\BQ_p^\nr}\OO(\D_\pst(T_{u_\dR,\underline v})) \isoarrow\overline K\otimes_K\OO(\D_\dR(T_{u_\dR,\underline v}))
\]
doesn't depend on $u_\dR$, and hence the canonical isomorphism $\gr^\W_\bullet\OO(\D_\pst(T_{u_\dR,\underline v}))\cong(\Lambda\otimes_{\BQ_p}\BQ_p^\nr)\otimes_{\BQ_p^\nr}\gr^\W_\bullet\OO(\D_\pst(U))$ is Hodge-filtered (after base-changing to $\Lambda\otimes_{\BQ_p}\overline K$). This implies that $\gr^\W_\bullet\OO(\D_\pst(T_{u_\dR,\underline v}))$ is ind-weakly admissible, and hence $\OO(\D_\pst(T_{u_\dR,\underline v}))$ is ind-weakly admissible by \cite[Proposition~4.4.4(iii)]{fontaine:semi-stables}.

\smallskip

Now since $\D_\pst(T_{u_\dR,\underline v})$ is weakly admissible, we obtain via Fontaine's $\mathsf V_\pst$ functor a de Rham $U$-torsor $T_{u_\dR,\underline v}$ whose image under $\D_\pst$ is $\D_\pst(T_{u_\dR,\underline v})$. It follows by a simple calculation that $\alpha([T_{u_\dR,\underline v}])=(u_\dR,\underline v)$, so $\alpha$ is surjective as claimed.
\end{proof}
\end{prop}

Proposition~\ref{prop:alpha_iso} proves Theorem~\ref{thm:representability} for $\HH^1_g$. The cases of $\HH^1_e$ and $\HH^1_f$ are covered by \cite[Proposition~1.4]{kim:tangential}.\qed
\appendix
\newcommand\g{\mathfrak g}
\newcommand\gbar{\overline\g}
\newcommand\f{\mathfrak f}
\renewcommand\r{\mathfrak r}
\newcommand\Disc{\mathbb D}
\newcommand\op{\mathrm{op}}
\newcommand\alg{\mathrm{alg}}
\newcommand\an{\mathrm{an}}
\newcommand\mM{\mathcal M}
\newtheorem{construction}[equation]{Construction}

\section{A canonical presentation for the weight-graded fundamental group}

In this appendix, we outline another proof of Theorems~\ref{weight-monodromy-pi1} and~\ref{semi-simplicity-l-adic} on mixedness and Frobenius-semisimplicity of the $\BQ_\ell$-pro-unipotent \'etale fundamental groups of smooth geometrically connected varieties $Y/K$. The aim here, following~\cite{hain:torelli}, is to write down an explicit and Galois-equivariant presentation of the associated $\W$-graded of the fundamental group, from which mixedness and Frobenius-semisimplicity are immediate.

For this section, we fix a smooth geometrically connected variety $Y$ over a characteristic $0$ field $K$ with fixed algebraic closure $\overline K$, and fix a simple normal crossings compactification $\overline Y$ of $Y$, with complementary divisor $D$. We write $I_1$ for the set of irreducible components of $D_{\overline K}=\bigcup_{i\in I_1}D_i$, and $I_2$ for the set of ordered pairs of distinct elements $i,j\in I_1$ such that $D_i\cap D_j\neq\emptyset$. Both the sets $I_1$ and $I_2$ carry an action of $G_K=\Gal(\overline K/K)$.

We write $\g=\g_b$ for the Lie algebra of the continuous $\BQ_\ell$-Mal\u cev completion of the profinite \'etale fundamental group $\pi_1^\et(Y_{\overline K},b)$ for some geometric basepoint $b$. There is a canonical isomorphism $\g^\ab\cong\HH_1^\et(Y_{\overline K},\BQ_\ell):=\HH^1_\et(Y_{\overline K},\BQ_\ell)^\dual$, which endows $\g$ with a \emph{weight filtration} $\W_\bullet\g$~\cite[Definition~1.5]{AMO}, namely the unique increasing filtration inducing the usual filtration on $\g^\ab=\HH_1^\et(Y_{\overline K},\BQ_\ell)$ for which the Lie bracket $[\cdot,\cdot]\colon\g\hatotimes\g\rightarrow\g$ is strict. This is the filtration described in Corollary \ref{lie-algebra-corollary}.

Now the associated graded $\gr^\W_\bullet\g$ is independent of $b$ up to canonical isomorphism, and hence inherits an action of $G_K$ from that on $Y_{\overline K}$. Our main theorem of this appendix gives an explicit $G_K$-equivariant presentation of this graded Lie algebra.

\begin{theorem}\label{thm:presentation}
Keep notation as above. Let $\f$ denote the free $\W$-graded pro-nilpotent Lie algebra over $\BQ_\ell$ generated by $\HH_1^\et(\overline Y_{\overline K},\BQ_\ell)$ in degree $-1$ and $\BQ_\ell(1)^{\oplus I_1}$ in degree $-2$. Let $\r\unlhd\f$ denote the (homogenous) ideal generated by the images of the following maps:
\begin{itemize}
	\item the map
	\[
	\HH_2^\et(\overline Y_{\overline K},\BQ_\ell) \rightarrow \gr^\W_{-2}\f = \bigwedge\nolimits^{\!\!2}\HH_1^\et(\overline Y_{\overline K},\BQ_\ell)\oplus\BQ_\ell(1)^{\oplus I_1}
	\]
	induced by the map dual to cup product and the maps dual to the Gysin maps $\HH^0_\et(D_i,\BQ_\ell)(-1)\rightarrow\HH^2_\et(\overline Y_{\overline K},\BQ_\ell)$;
	\item the map
	\[
	\bigoplus_{i\in I_1}\HH_1^\et(D_i,\BQ_\ell)(1)\rightarrow\gr^\W_{-3}\f
	\]
	given by taking the commutator of the maps $\HH_1^\et(D_i,\BQ_\ell)\rightarrow\HH_1^\et(\overline Y_{\overline K},\BQ_\ell)$ and the $i$th coproduct inclusion $\BQ_\ell(1) \rightarrow \BQ_\ell(1)^{\oplus I_1}$; and
	\item the map
	\[
	\BQ_\ell(2)^{\oplus I_2} \rightarrow \gr^\W_{-4}\f
	\]
	given by taking the commutator of the $i$th and $j$th coproduct inclusions $\BQ_\ell(1) \rightarrow \BQ_\ell(1)^{\oplus I_1}$ whenever $D_i$ and $D_j$ intersect.
\end{itemize}

Then there is a canonical $G_K$-equivariant and $\W$-graded isomorphism
\[
\gr^\W_\bullet\g\cong\f/\r\,.
\]
\end{theorem}

For us, the point of Theorem~\ref{thm:presentation} is that it allows us to automatically deduce properties of \'etale fundamental groupoids of smooth varieties from the corresponding properties of low-dimensional (co)homology of smooth proper varieties. For instance, the cohomology of a smooth proper variety is pure in degrees $\leq2$ (Proposition~\ref{Hi-mixed} plus Chow's Lemma and the Lefschetz Hyperplane Theorem) and Frobenius-semisimple in degrees $\leq1$ (Proposition~\ref{semisimple-H1}). This then implies the corresponding properties for the fundamental groupoid.

\begin{corollary}[Weight--monodromy and Frobenius-semisimplicity for the fundamental group]
Suppose that $Y$ is a smooth geometrically connected variety over a $p$-adic field $K$. Then $\OO(\pi_1^{\BQ_\ell}(Y_{\overline K};x,y))$ and $\Lie(\pi_1^{\BQ_\ell}(Y_{\overline K},z))$ are mixed and Frobenius-semisimple for all rational points $x,y,z\in Y(K)$ (or tangential points). Here, the weight filtration on $\OO(\pi_1^{\BQ_\ell}(Y_{\overline K};x,y))$ is dual to the weight filtration on $\OO(\pi_1^{\BQ_\ell}(Y_{\overline K};x,y))^\dual=\varprojlim_n\ZZ_\ell\llbrack\pi_1^\ell(Y_{\overline K};x,y)\rrbrack\otimes_{\ZZ_\ell}\BQ_\ell$ as defined in Definition~\ref{weight-defn}, and the weight filtration on $\Lie(\pi_1^{\BQ_\ell}(Y_{\overline K};z))$ is the one defined above.
\begin{proof}
We will argue the $\ell\neq p$ and $\ell=p$ cases simultaneously, recalling in the latter case that $\OO(\pi_1^{\BQ_\ell}(Y_{\overline K};x,y))^\dual$ is de Rham (Remark~\ref{de-rham-remark}), and hence so too is $\Lie(\pi_1^{\BQ_\ell}(Y_{\overline K},z))$. By Definition~\ref{def:canonical_splitting}, it suffices to prove that $\gr^\W_\bullet\OO(\pi_1^{\BQ_\ell}(Y_{\overline K}))^\dual$ and $\gr^\W_\bullet\Lie(\pi_1^{\BQ_\ell}(Y_{\overline K}))$ are pure and Frobenius-semisimple. Note that these objects are independent of the points $x,y,z$ up to canonical $G_K$-equivariant isomorphism, and that $\gr^\W_\bullet\OO(\pi_1^{\BQ_\ell}(Y_{\overline K}))^\dual$ is the completed universal enveloping algebra of $\gr^\W_\bullet\Lie(\pi_1^{\BQ_\ell}(Y_{\overline K}))=\gr^\W_\bullet\g$~\cite[Example~A.3.8]{alex-netan}. Since the class of pure and Frobenius-semisimple representations is closed under cokernels, it suffices to show that $\gr^\W_\bullet\g$ is pure and Frobenius-semisimple.

Now, for purity, it suffices by Theorem~\ref{thm:presentation} to show that $\HH_1^\et(\overline Y_{\overline K},\BQ_\ell)$, $\BQ_\ell(1)^{\oplus I_1}$, $\HH_2^\et(\overline Y_{\overline K},\BQ_\ell)$, $\bigoplus_{i\in I_1}\HH_1^\et(D_i,\BQ_\ell)(1)$ and $\BQ_\ell(2)^{\oplus I_2}$ are pure of weights $-1$, $-2$, $-2$, $-3$ and $-4$, respectively. Purity of the second and fifth of these is immediate; purity of the first, third and fourth follows from Proposition~\ref{Hi-mixed}.

For Frobenius-semisimplicity, it suffices to prove that $\HH_1^\et(\overline Y_{\overline K},\BQ_\ell)$ and $\BQ_\ell(1)^{\oplus I_1}$ are Frobenius-semisimple. The second is obvious; the first is Proposition~\ref{semisimple-H1}.
\end{proof}
\end{corollary}

The rest of this section is devoted to a proof of Theorem~\ref{thm:presentation}.

\subsection{The complex case}

To begin with, suppose that $K=\CC$, so that $\g$ is the base change to $\BQ_\ell$ of the Lie algebra of the $\BQ$-Mal\u cev completion of the Betti fundamental group of $Y=Y(\CC)$. Henceforth, we will let $\g$ instead denote this $\BQ$-linear pro-nilpotent Lie algebra, and prove the $\BQ$-linear analogue of Theorem~\ref{thm:presentation}, using Betti homology in place of \'etale homology. We fix an orientation on $\CC$, which provides us with a generator of $\BQ(1)$.

Throughout the proof, we will use a certain natural topological compactification $\widehat Y$ of $Y$, whose construction is due to A'Campo \cite{a'campo:zeta_de_monodromie}.

\begin{construction}\label{cons:acampo}
For each $i\in I_1$, we let $\widehat Y_i$ denote the real (oriented) blowup of $\overline Y$ along $D_i$, so that there is a proper map $\widehat Y_i\twoheadrightarrow\overline Y$ of manifolds with corners which is an isomorphism away from $D_i$ and is an oriented $S^1$-bundle over $D_i$. We let $\widehat Y$ denote the fibre product of $(\widehat Y_i)_{i\in I_1}$ over $\overline Y$, so that the inclusion $Y\hookrightarrow\overline Y$ factors through $Y\hookrightarrow\widehat Y\twoheadrightarrow\overline Y$. The first of these maps is a homotopy equivalence, so that $\g$ is canonically identified with the Lie algebra of the $\BQ$-pro-unipotent Betti fundamental group of $\widehat Y$.
\end{construction}

To begin the proof of Theorem~\ref{thm:presentation}, we first construct a graded Lie algebra map $\f\rightarrow\gr^\W_\bullet\g$, for which we need to construct maps $\HH_1(\overline Y,\BQ)\rightarrow\gr^\W_{-1}\g$ and $\BQ(1)^{\oplus I_1}\rightarrow\gr^\W_{-2}\g$. For the first of these, we take the inverse of the isomorphism $\gr^\W_{-1}\g\isoarrow\gr^\W_{-1}\HH_1(Y,\BQ)=\HH_1(\overline Y,\BQ)$ arising from the definition of $\W_\bullet$. For the second, we choose for each $i\in I_1$ a point $x_i$ in $D_i$ not lying in any other component of $D$. The fibre of $\widehat Y\twoheadrightarrow\overline Y$ over $x_i$ is an oriented circle $\gamma_i$, and we write $\log(\gamma_i)\in\gr^\W_{-2}\g$ for the element determined by this loop. The desired map $\BQ(1)^{\oplus I_1}\rightarrow\gr^\W_{-2}\g$ is then the map taking the $i$th basis vector to $\log(\gamma_i)$. It is easy to see that the free homotopy class of the loop $\gamma_i$, and hence the map $\BQ(1)^{\oplus I_1}\rightarrow\gr^\W_{-2}\g$, is independent of the choice of $x_i$.

% note that $\widehat Y\twoheadrightarrow\overline Y$ is an oriented $S^1$-bundle over $D_i^\circ:=D_i\setminus\bigcup_{j\neq i}D_j$, and hence we have a map $\BQ(1)\rightarrow\g$ (well-defined up to conjugacy) given by the inclusion of the fibre $S^1$. The image of this map is contained in $\W_{-2}$ by construction, so induces a well-defined map $\BQ(1)\rightarrow\gr^\W_{-2}\g$; the sum of these maps is the desired map $\BQ(1)^{\oplus I_1}\rightarrow\gr^\W_{-2}\g$.

Changing notation slightly from Theorem~\ref{thm:presentation}, we write $\r$ for the kernel of the map $\f\rightarrow\gr^\W_\bullet\g$ defined above. We thus want to prove the following.

\begin{enumerate}[label=\normalfont{\emph{\alph*})},ref=\normalfont{(\emph{\alph*})}]
	\item\label{thmpart:generators} The map $\f\rightarrow\gr^\W_\bullet\g$ is surjective.%; equivalently the map $\f^\ab\rightarrow\left(\gr^\W_\bullet\g\right)^\ab$ is surjective.
	\item\label{thmpart:relations_satisfied} The images of the maps $\HH_2(\overline Y,\BQ)\rightarrow\gr_{-2}^\W\f$, $\bigoplus_{i\in I_1}\HH_1(D_i,\BQ)(1)\rightarrow\gr^\W_{-3}\f$ and $\BQ(2)^{\oplus I_2}\rightarrow\gr^\W_{-4}\f$ defined as in Theorem~\ref{thm:presentation} are contained in~$\r$.
	\item\label{thmpart:relations_suffice} The images of these maps generate $\r$ as an ideal; equivalently they generate $\r/[\f,\r]$ as a vector space.
\end{enumerate}

\begin{proof}[Proof of~\ref{thmpart:generators}]
We may lift the $\W$-graded map $\f\rightarrow\gr^\W_\bullet\g$ to a $\W$-filtered map $\f\rightarrow\g$. Now it is easy to see that the map $\f^\ab\rightarrow\g^\ab=\HH_1(Y,\BQ)=\HH_1(\widehat Y,\BQ)$ is surjective and $\W$-strict: it surjects onto the quotient $\gr^\W_{-1}\HH_1(Y,\BQ)=\HH_1(\overline Y,\BQ)$ and the kernel is generated by the classes of the loops $\gamma_i$, situated in degree $-2$. It follows that the map $\f\rightarrow\g$ is surjective and $\W$-strict; passing to the associated $\W$-gradeds shows that $\f\twoheadrightarrow\gr^\W_\bullet\g$ is surjective, as desired.
\end{proof}

%\begin{proof}[Proof of~\ref{thmpart:generators}]
%It follows from Lemma~\ref{lem:deligne_splitting} that the natural map $\left(\gr^\W_\bullet\g\right)^\ab\rightarrow\gr^\W_\bullet(\g^\ab)=\gr^\W_\bullet\HH_1(Y,\BQ)$ is an isomorphism, so it suffices to prove that the map $\f^\ab\rightarrow\gr^\W_\bullet\HH_1(Y,\BQ)$ is surjective. Surjectivity in degree $-1$ is obvious; in degree $-2$ it follows from the well-known fact that the kernel of $\HH_1(Y,\BQ)\rightarrow\HH_1(\overline Y,\BQ)$ is generated by the classes of loops around the boundary components $D_i$.
%\end{proof}

The proof of~\ref{thmpart:relations_satisfied} is purely topological, showing that the claimed relations in $\gr^\W_\bullet\g$ are witnessed by certain immersed surfaces in $\widehat Y$.

\begin{proof}[Proof of~\ref{thmpart:relations_satisfied}]
\textit{Weight $-4$}: Suppose that $D_i\cap D_j\neq\emptyset$, and choose a point $x_{ij}$ in their intersection not lying in any other component of $D$. The fibre of $\widehat Y\twoheadrightarrow\overline Y$ over $x_{ij}$ is diffeomorphic to $S^1\times S^1$, with the two product projections corresponding to the projections $\widehat Y\twoheadrightarrow\widehat Y_i$, $\widehat Y\twoheadrightarrow\widehat Y_j$, respectively. It is easy to see that the two standard loops $\gamma_1=S^1\times\{\ast\}$ and $\gamma_2=\{\ast\}\times S^1$ are freely homotopic to the loops $\gamma_i$ and $\gamma_j$ respectively. Since the fundamental group of $S^1\times S^1$ is commutative, we have $\bigl[\log(\gamma_i),\log(\gamma_j)\bigr]=0$ in $\gr^\W_{-4}\g$. This says that the image of $\BQ(2)^{\oplus I_2}\rightarrow\gr^\W_{-4}\f$ is contained in $\r$, as desired.

\textit{Weight $-3$}: Consider any element $[\gamma]\in\HH_1(D_i,\ZZ)$, which we may choose to be represented by an immersed loop $\gamma\colon S^1\rightarrow D_i^\circ:=D_i\setminus\bigcup_{j\neq i}D_j$. The fibre of $\widehat Y\twoheadrightarrow\overline Y$ over $\gamma$ is an oriented $S^1$-bundle over $S^1$, and hence diffeomorphic to a torus $S^1\times S^1$. The immersion $S^1\times S^1\rightarrow\widehat Y$ takes the standard loops $\gamma_1$ and $\gamma_2$ to (a lift of) $\gamma$ and $\gamma_i$, respectively. As before, this yields the relation $\bigl[\log(\gamma),\log(\gamma_i)\bigr]=0$ in $\gr^\W_{-3}\g$, so that the image of $\HH_1(D_i,\BQ)(1)\rightarrow\gr^\W_{-3}\f$ is contained in $\r$, as desired.

\textit{Weight $-2$}: Consider any class $[\overline\Sigma]\in\HH_2(\overline Y,\ZZ)$, represented by an immersion $\iota\colon\overline\Sigma\rightarrow\overline Y$ with $\overline\Sigma$ a compact connected oriented surface which meets $D$ transversely at smooth points. We let $\iota^{-1}D=\{x_1,\dots,x_n\}$ denote the preimage of $D$ in $\overline\Sigma$ and for each $1\leq j\leq n$ write $D_{\iota(j)}$ for the component containing $x_i$ and $\epsilon_j\in\{\pm1\}$ for the local intersection number of $\overline\Sigma$ and $D$ at $x_j$.

The fundamental group of $\Sigma:=\overline\Sigma\setminus\{x_1,\dots,x_n\}$ has a presentation of the form
\[
\pi_1(\Sigma)=\left\langle(a_j)_{j=1}^g,(b_j)_{j=1}^g,(c_j)_{j=1}^n\::\:\prod_{j=1}^g[a_j,b_j]\cdot\prod_{j=1}^nc_j=1\right\rangle\,,
\]
where the $a_j$ and $b_j$ form a symplectic basis of $\HH_1(\overline\Sigma,\ZZ)$ and each $c_j$ is freely homotopic to a positively oriented loop around $x_j$. It follows that we have
\[
\sum_{j=1}^g\bigl[[\iota_*(a_j)],[\iota_*(b_j)]\bigr]+\sum_{j=1}^n\epsilon_j[\gamma_{\iota(j)}]=0
\]
in $\gr^\W_{-2}\g$. But $\sum_{j=1}^g[\iota_*(a_j)]\wedge[\iota_*(b_j)]$ is the image of $[\overline\Sigma]$ under the cup coproduct map, while the image of $[\overline\Sigma]$ under the dual Gysin map associated to $D_i$ is $[D_i\cap\overline\Sigma]=\sum_{\iota(j)=i}\epsilon_j$. Thus, the above identity establishes that the image of $[\overline\Sigma]$ in $\gr^\W_{-2}\g$ is zero, as desired.
\end{proof}

The proof of~\ref{thmpart:relations_suffice} is somewhat more technical, and follows \cite[\S5]{hain:torelli}. We will control the quotient $\r/[\f,\r]$ using Lie algebra homology, specifically the exact sequence
\begin{equation}\label{eq:graded_hochschild-serre}
0 \rightarrow \HH_2(\gr^\W_\bullet\g) \rightarrow \frac\r{[\f,\r]} \rightarrow \f^\ab \rightarrow \left(\gr^\W_\bullet\g\right)^\ab \rightarrow 0
\end{equation}
of low-degree terms in the Hochschild--Serre spectral sequence associated to the extension $0\rightarrow\r\rightarrow\f\rightarrow\gr^\W_\bullet\g\rightarrow0$ of pro-nilpotent Lie algebras. We control the leftmost term of this sequence using the following proposition, whose proof uses the Hodge theory of the pro-unipotent fundamental group in an essential way.

\begin{prop}\label{prop:magic_MHS}
\leavevmode
\begin{enumerate}[label=\normalfont{\emph{\roman*})}, ref=\normalfont{(\emph{\roman*})}]
	\item\label{proppart:gradings_okay} The natural map $\gr^\W_\bullet\HH_2(\g)\rightarrow\HH_2(\gr^\W_\bullet\g)$ is an isomorphism.
	\item\label{proppart:H_2_surjects} The natural map $\HH_2(Y,\BQ)\rightarrow\HH_2(\g)$ is surjective and $\W$-strict.
\end{enumerate}
Here, the $\W$-filtration on $\HH_2(\g)$ is the natural one whereby $\W_{-k}\HH_2(\g)$ is spanned by the classes of ``relations of weight $-k$''.
% (for a precise definition, see the proof).
%whereby $\W_{-k}\HH_2(\g)$ is spanned by the classes of relations
%\[
%\sum_i[x_{i,1},[x_{i,2},\dots[x_{i,m_i-1},x_{i,m_i}]\dots]]
%\]
%where $x_{i,1}\otimes x_{i,2}\otimes\dots\otimes x_{i,m_i-1}\otimes x_{i,m_i}\in\W_{-k}\g^{\hatotimes m_i}$ for all $i$.
\begin{proof}
%Let $\tilde\f$ denote the free pro-nilpotent Lie algebra on $\g$ and $\tilde\r$ the kernel of the evident surjection $\tilde\f\twoheadrightarrow\g$. We endow $\tilde\f$ with the $\W$-filtration induced from that on~$\g$. The Hochschild--Serre spectral sequence
%\begin{equation}\label{eq:filtered_hochschild-serre}
%0 \rightarrow \HH_2(\g) \rightarrow \frac{\tilde\r}{[\tilde\f,\tilde\r]} \rightarrow \tilde\f^\ab \rightarrow \g^\ab \rightarrow 0
%\end{equation}
%associated to the extension $0\rightarrow\tilde\r\rightarrow\tilde\f\rightarrow\g\rightarrow0$ realises $\HH_2(\g)$ as a subquotient of $\tilde\f$, such that its natural filtration is the subquotient filtration.

$\g$ carries a pro-mixed Hodge structure, compatible with the Lie bracket, whose weight filtration is $\W_\bullet$ \cite{hain:mhs}. The Deligne splitting provides a canonical splitting of the weight filtration on $\g_\CC$ compatible with the Lie bracket. It follows that the natural map $\gr^\W_\bullet\HH_2(\g)\rightarrow\HH_2(\gr^\W_\bullet\g)$ becomes an isomorphism after tensoring with $\CC$, and hence is already an isomorphism. This proves~\ref{proppart:gradings_okay}.

The surjectivity in~\ref{proppart:H_2_surjects} is well-known; $\W$-strictness follows from the fact that it is a morphism of mixed Hodge structures~\cite[Theorem~11.7]{carlson-hain}.
\end{proof}
\end{prop}

\begin{proof}[Proof of \ref{thmpart:relations_suffice}]
It follows from Proposition~\ref{prop:magic_MHS} and sequence~\eqref{eq:graded_hochschild-serre} that the grading on $\r/[\f,\r]$ is supported in degrees $-2$, $-3$ and $-4$. We thus want to show surjectivity of the three maps
\begin{align}
\HH_2(\overline Y,\BQ) &\rightarrow \gr^\W_{-2}\left(\r/[\f,\r]\right)\,, \label{map:wt2}\tag{$\ast_2$}\\
\bigoplus_{i\in I_1}\HH_1(D_i,\BQ)(1) &\rightarrow \gr^\W_{-3}\left(\r/[\f,\r]\right) = \gr^\W_{-3}\HH_2(\g)\,, \label{map:wt3}\tag{$\ast_3$}\\
\BQ(2)^{\oplus I_2} &\rightarrow \gr^\W_{-4}\left(\r/[\f,\r]\right) = \gr^\W_{-4}\HH_2(\g)\,, \label{map:wt4}\tag{$\ast_4$}
\end{align}
which we address one by one.

\textit{Weight $-3$}: It suffices to prove that \eqref{map:wt3} is equal to the composite
\begin{equation}\label{eq:composite_wt3}\tag{$\dag_3$}
\bigoplus_{i\in I_1}\HH_1(D_i,\BQ)(1)\twoheadrightarrow\gr^\W_{-3}\HH_2(Y,\BQ)\twoheadrightarrow\gr^\W_{-3}\HH_2(\g)
\end{equation}
where the first map is the map arising from the weight spectral sequence (the Leray spectral sequence associated to $\widehat Y\rightarrow\overline Y$) and the second map is the map from Proposition~\ref{prop:magic_MHS}\ref{proppart:H_2_surjects}. To do this, consider an element $[\gamma]\in\HH_1(D_i,\ZZ)$, which we represent by an immersed loop $\gamma\colon S^1\rightarrow D_i^\circ$ as in the proof of \ref{thmpart:relations_satisfied}. Pulling back $\widehat Y\rightarrow\overline Y$ along $\gamma$ yields an oriented $S^1$-bundle over $S^1$, so that $\gamma$ lifts to an immersion $\iota\colon S^1\times S^1\rightarrow\widehat Y$. Considering the map on Leray spectral sequences induced by the commuting square
\begin{center}
\begin{tikzcd}
S^1\times S^1 \arrow{r}{\iota}\arrow{d} & \widehat Y \arrow{d} \\
S^1 \arrow{r}{\gamma} & \overline Y
\end{tikzcd}
\end{center}
shows that the image of the class $[\gamma](1)$ under the first map of~\eqref{eq:composite_wt3} is the pushforward $\iota_*[S^1\times S^1]\in\gr^\W_{-3}\HH_2(Y,\BQ)$ of the orientation class on $S^1\times S^1$.

On the other hand, if we write $\g_{S^1\times S^1}$ for the Lie algebra of the $\BQ$-Mal\u cev completion of the fundamental group of $S^1\times S^1$ (so a $2$-dimensional vector group), then $\HH_2(\g_{S^1\times S^1})$ is one-dimensional, spanned by the class of the relation $[\log(\gamma_1),\log(\gamma_2)]$ with $\gamma_1$ and $\gamma_2$ the standard generators of $\pi_1(S^1\times S^1)$. Note that this generating class is the image of the orientation class on $S^1\times S^1$ under the natural map $\HH_2(S^1\times S^1,\BQ)\twoheadrightarrow\HH_2(\g_{S^1\times S^1})$. Now the image of $[\gamma](1)$ under~\eqref{map:wt3} is the pushforward $\iota_*([\log(\gamma_1),\log(\gamma_2)])\in\gr^\W_{-3}\HH_2(\g)$, so that from the commuting square
\begin{center}
\begin{tikzcd}
\HH_2(S^1\times S^1,\BQ) \arrow{r}{\iota_*}\arrow[two heads]{d} & \HH_2(Y,\BQ) \arrow[two heads]{d} \\
\HH_2(\g_{S^1\times S^1}) \arrow{r}{\iota_*} & \HH_2(\g)
\end{tikzcd}
\end{center}
we see that $[\gamma](1)$ has the same image under~\eqref{map:wt3} and~\eqref{eq:composite_wt3}. Thus these two maps are equal, as desired.

\smallskip

\textit{Weight $-4$}: Let $\widetilde I_2$ denote the set of ordered pairs of distinct indices $i,j\in I_1$ together with an irreducible component $D_{ij}$ of $D_i\cap D_j$. It suffices to prove that the composite of~\eqref{map:wt4}, precomposed with the natural surjection $\BQ(2)^{\oplus\widetilde I_2}\twoheadrightarrow\BQ(2)^{\oplus I_2}$, is equal to the composite
\begin{equation}\label{eq:composite_wt4}\tag{$\dag_4$}
\BQ(2)^{\oplus\widetilde I_2}\twoheadrightarrow\gr^\W_{-4}\HH_2(Y,\BQ)\twoheadrightarrow\gr^\W_{-4}\HH_2(\g)
\end{equation}
where the first map is the map arising from the weight spectral sequence and the second map is the map from Proposition~\ref{prop:magic_MHS}\ref{proppart:H_2_surjects}. To do this, consider an element $(i,j,D_{ij})\in\widetilde I_2$, and choose a point $x_{ij}\in D_{ij}$ not lying in any other component of $D$. The fibre of $\widehat Y\rightarrow\overline Y$ over $x_{ij}$ is a torus $S^1\times S^1$, with the two standard loops $\gamma_1$ and $\gamma_2$ being freely homotopic to the loops $\gamma_i$ and $\gamma_j$, as in the proof of~\ref{thmpart:relations_satisfied}. As above, we find that the image of $[x_{ij}](2)$ under the first map of~\eqref{eq:composite_wt4} is the pushforward $\iota_*([S^1\times S^1])$ of the orientation class under the inclusion $\iota\colon S^1\times S^1\rightarrow\widehat Y$, while its image under~\eqref{map:wt4} is $\iota_*([\log(\gamma_1),\log(\gamma_2)])$. The same argument establishes that $\iota_*([S^1\times S^1])$ maps to $\iota_*([\log(\gamma_1),\log(\gamma_2)])$ under the map $\gr^\W_{-4}\HH_2(Y,\BQ)\rightarrow\gr^\W_{-4}\HH_2(\g)$, and hence we are done also in this case.

\smallskip

\textit{Weight $-2$}: There are two claims to prove here: that map~\eqref{map:wt2} restricts to the surjection $\gr^\W_{-2}\HH_2(Y,\BQ)\twoheadrightarrow\gr^\W_{-2}\HH_2(\g)$ from Proposition~\ref{prop:magic_MHS}\ref{proppart:H_2_surjects}; and that it induces a surjection from $\HH_2(\overline Y,\BQ)/\gr^\W_{-2}\HH_2(Y,\BQ)$ to the kernel of $\f^\ab\twoheadrightarrow\left(\gr^\W_\bullet\g\right)^\ab$. For the first of these claims, we proceed as above. Pick a class $[\Sigma]\in\HH_2(Y,\BQ)$, represented by a immersed closed oriented surface $\iota\colon\Sigma\rightarrow Y$, and write $\g_\Sigma$ for the Lie algebra of the $\BQ$-Mal\u cev completion of the fundamental group of $\Sigma$. Thus $\HH_2(\g_\Sigma)$ is spanned by the class of the relation $\log\left(\prod_{j=1}^g[a_j,b_j]\right)$, where the $a_j$ and $b_j$ are the standard generators of the fundamental group of a closed oriented surface of genus $g$. Again, this class is the image of the orientation class of $\Sigma$ under the map $\HH_2(\Sigma,\BQ)\twoheadrightarrow\HH_2(\g_\Sigma)$, and hence the image of $[\Sigma]\in\HH_2(Y,\BQ)$ under~\eqref{map:wt2} is the same as its image under the surjection $\gr^\W_{-2}\HH_2(Y,\BQ)\twoheadrightarrow\gr^\W_{-2}\HH_2(\g)$ from Proposition~\ref{prop:magic_MHS}\ref{proppart:H_2_surjects}. This establishes the first claim.

For the second claim, we note that the low-degree terms of the Leray spectral sequence provides an exact sequence
\[
0 \rightarrow \gr^\W_{-2}\HH_2(Y,\BQ) \rightarrow \HH_2(\overline Y,\BQ) \rightarrow \BQ(1)^{\oplus I_1} \rightarrow \gr^\W_{-2}\HH_1(Y,\BQ) \rightarrow 0
\]
where the central map is dual to the sum of the Gysin maps and the right-hand map is the weight $-2$ part of $\f^\ab\twoheadrightarrow\left(\gr^\W_\bullet\g\right)^\ab$. This directly establishes the second claim.
\end{proof}

\subsection{The general case}
Finally, we deduce Theorem~\ref{thm:presentation} from the special case $K=\CC$. By the Lefschetz principle, it suffices to prove this when $\overline K$ admits an embedding $\overline K\hookrightarrow\CC$. Using this embedding, we obtain from the previous section a presentation of $\gr^\W_\bullet\g$ of the desired form, which we wish to show is $G_K$-equivariant. The only part of this which is non-obvious is $G_K$-equivariance of the map $\BQ_\ell(1)^{\oplus I_1}\rightarrow\gr^\W_{-2}\g$.

To do this, we pick for each $i$ a point $x_i\in D_i(\overline K)$ not lying in any other component, and write $\widehat\OO_{\overline Y_{\overline K},x_i}$ for the completed local ring at $x_i$. We choose a local parameter $t_i\in\OO_{\overline Y_{\overline K},x_i}$ cutting out $D_i$, and denote the evident morphism $\Spec\left(\widehat\OO_{\overline Y_{\overline K},x_i}\left[1/t_i\right]\right)\rightarrow Y$ by $\gamma_i^\alg$. Abhyankar's Lemma \cite[Proposition~XIII.5.2]{SGA1} ensures that the \'etale fundamental group of $\Spec\left(\widehat\OO_{\overline Y_{\overline K},x_i}\left[1/t_i\right]\right)$ is canonically isomorphic to $\widehat\ZZ(1)$, and hence pushforward yields an outer homomorphism $\widehat\ZZ(1)\rightarrow\pi_1^\et(Y_{\overline K})$. The key result here states that the maps $\gamma_i^\alg$ are algebraic avatars of the loops $\gamma_i$.

\begin{lemma}\label{lem:algebraic_loops}
For every embedding $\overline K\hookrightarrow\CC$, the outer homomorphism $\gamma_{i,*}^\alg\colon\widehat\ZZ(1)\rightarrow\pi_1^\et(Y_{\overline K})$ is identified with the profinite completion of the outer homomorphism $\gamma_{i,*}\colon\ZZ(1)\rightarrow\pi_1(Y(\CC))$ under the usual identification $\pi_1^\et(Y_{\overline K})\cong\widehat\pi_1(Y(\CC))$.
\end{lemma}

\begin{corollary}
The outer homomorphisms $\gamma_{i,*}^\alg\colon\widehat\ZZ(1)\rightarrow\pi_1^\et(Y_{\overline K})$ are independent of the choice of points $x_i$, and are Galois-equivariant in the sense that $\sigma\circ\gamma_{i,*}^\alg\circ\sigma^{-1}=\gamma_{\sigma(i),*}^\alg$ for all $\sigma\in G_K$.
\begin{proof}
If $\overline K$ admits an embedding in $\CC$, then the first assertion is immediate; in general use the Lefschetz Principle to reduce to this case. The second follows easily from the first.
\end{proof}
\end{corollary}

\begin{proof}[Proof of Lemma~\ref{lem:algebraic_loops}]
It suffices to prove this in the case $K=\CC$. Denote by $\OO_{\overline Y^\an,x_i}$ (resp.\ $\OO_{\overline Y^\an,x_i}$) the ring of germs of holomorphic functions on $\overline Y^\an$ (resp.\ $Y^\an$) at $x_i$, i.e.\ the direct limit of the rings $\OO_{\overline Y^\an}(\overline U_i')$ (resp.\ $\OO_{Y^\an}(\overline U_i'\cap Y^\an)$) as $\overline U_i'$ runs over open neighbourhoods of $x_i$ in $\overline Y_i$. Fix one such neighbourhood $\overline U_i$, and assume that $\overline U_i$ is biholomorphic to a polydisc $\Disc^n$ in such a way that $\overline U_i\cap D_i=\{0\}\times\Disc^{n-1}$ and $\overline U_i$ meets no other components of $D_i$. We write $U_i=\overline U_i\cap Y^\an$. There is thus a commuting diagram
\begin{center}
\begin{tikzcd}[column sep=small]
{} & \Spec\left(\OO_{Y^\an,x_i}\right) \arrow{r}\arrow{d} & \Spec\left(\OO_{Y^\an}(U_i)\right) \arrow{d} \\
\Spec\left(\widehat\OO_{\overline Y,x_i}[1/t_i]\right) \arrow{r} & \Spec\left(\OO_{\overline Y^\an,x_i}[1/t_i]\right) \arrow{r} & Y
\end{tikzcd}
\end{center}
of schemes, where the composite along the bottom row is $\gamma_i^\alg$. We will deduce the lemma by showing that applying the functor $\pi_1^\et$ to this diagram yields the diagram
\begin{center}
\begin{tikzcd}
{} & \widehat\ZZ(1) \arrow[equals]{r}{(2)}\arrow[equals]{d}[swap]{(4)} & \widehat\ZZ(1) \arrow{d}{\widehat\gamma_{i,*}}[swap]{(1)} \\
\widehat\ZZ(1) \arrow[equals]{r}{(3)} & \widehat\ZZ(1) \arrow{r}{\gamma_{i,*}^\alg} & \pi_1^\et(Y)
\end{tikzcd}
\end{center}
of outer homomorphisms. We justify this in several steps.

(1) Note that by construction $U_i$ is homotopy equivalent to a complex punctured disc, and so has fundamental group $\ZZ(1)$; moreover the outer homomorphism $\ZZ(1)\rightarrow\pi_1(Y^\an)$ induced by the inclusion $U_i\hookrightarrow Y^\an$ is the pushforward map $\gamma_{i,*}$. But there is an equivalence of categories between finite coverings of $U_i$ and finite \'etale algebras over $\OO_{Y^\an}(U_i)$, given by taking a finite covering $\pi\colon V\rightarrow U_i$ to the algebra $\OO(V^\an)$ where $V^\an$ is endowed with the unique complex structure making $\pi$ holomorphic. It thus follows that $\pi_1^\et\left(\Spec\left(\OO_{Y^\an}(U_i)\right)\right)=\widehat\ZZ(1)$, and that the outer homomorphism $\widehat\ZZ(1)\rightarrow\pi_1^\et(Y)$ induced by the scheme morphism $\Spec\left(\OO_{Y^\an}(U_i)\right)\rightarrow Y$ is equal to $\widehat\gamma_{i,*}$.

(2) A similar argument as for (1) establishes that $\Spec\left(\OO_{Y^\an},x_i\right)$ is the inverse limit of the profinite completions of the fundamental groups of $\overline U_i'\cap Y^\an$, i.e.\ $\widehat\ZZ(1)$, and that map (2) is the identity.

(3) The ring $\OO_{\overline Y^\an,x_i}$ is strictly Henselian, so Abhyankar's Lemma again implies that $\pi_1^\et\left(\Spec\left(\OO_{\overline Y^\an,x_i}[1/t_i]\right)\right)=\widehat\ZZ(1)$, and that map (3) is the identity.

(4) We know that the finite \'etale algebras over both $\OO_{\overline Y^\an,x_i}[1/t_i]$ and $\OO_{Y^\an,x_i}$ are given by adjoining roots of $t_i$: the former via Abhyankar's Lemma and the latter via identifying these algebras with covers of $U_i$. It follows that (4) is the identity.
\end{proof}

\bibliographystyle{alpha}
\bibliography{semisimplicity_references}

\end{document}